\documentclass[final]{siamltex1213pdflatex}

\usepackage{amsfonts,amsmath}
\usepackage{epsfig}
\usepackage{graphicx}

\newtheorem{remark}[theorem]{Remark}

\newcommand{\norm}[1]{\left\Vert#1\right\Vert}

\newcommand{\Real}{\mathbb R}

\newcommand{\Exp}{\mathrm{Exp}}
\newcommand{\Log}{\mathrm{Log}}

\renewcommand{\thefootnote}{\fnsymbol{footnote}} 	
\def\argmin{\mathop{\rm argmin}}
\def\argmax{\mathop{\rm argmax}}

\def\tr{\mbox{trace}}
\def\half{\frac{1}{2}}

\def\diag{\mbox{diag}}

\def\symp{\mbox{Sym}^+}
\def\sym{\mbox{Sym}}
\def\SO{\mbox{SO}}
\def\O{\mbox{O}}
\def\GL{\mbox{GL}}
\def\asym{\mathfrak{so}}
\def\so{\mathfrak{so}}
\def\diagp{\mbox{Diag}^+}
\def\diag{\mbox{Diag}}

\def\av{\mathbf a}

\def\ev{\mathbf e}

\def\gf{\mathfrak g}
\def\ad{\mbox{ad}}

\newcommand{\Dc}{\mathcal{D}}
\newcommand{\Ec}{\mathcal{E}}

\newcommand{\Jc}{\mathcal{J}}
\newcommand{\Kc}{\mathcal{K}}

\newcommand{\Rc}{\mathcal{R}}
\newcommand{\Sc}{\mathcal{S}}

\def\1v{\mathbf 1}
\def\0v{\mathbf 0}
\def\Id{ I}

\newcommand{\boldsig}{\mbox{\boldmath $\sigma$ \unboldmath}
           \mbox{\hspace{-.05in}}}

\title{Scaling-rotation distance and interpolation of symmetric positive-definite matrices \thanks{This work was
supported by NIH grant R21EB012177 and NSF grant DMS-1307178.}}

\author{Sungkyu Jung\footnotemark[2]
\and Armin Schwartzman\footnotemark[3]
\and David Groisser\footnotemark[4]}

\begin{document}
\maketitle
\renewcommand{\thefootnote}{\fnsymbol{footnote}}
\footnotetext[2]{Department of Statistics, University of Pittsburgh, Pittsburgh, PA 15260, USA (\email{sungkyu@pitt.edu}).}
\footnotetext[3]{Department of Statistics, North Carolina State University, Raleigh, NC 27695, USA (\email{aschwar@ncsu.edu}).}
\footnotetext[4]{Department of Mathematics, University of Florida, Gainesville, FL 32611, USA (\email{groisser@ufl.edu}).}
\renewcommand{\thefootnote}{\arabic{footnote}}

\slugger{simax}{xxxx}{xx}{x}{x--x}

\begin{abstract}
We introduce a new geometric framework for the set of symmetric positive-definite (SPD) matrices, aimed to characterize deformations of SPD matrices by individual scaling of eigenvalues and rotation of eigenvectors of the SPD matrices.
To characterize the  deformation, the eigenvalue-eigenvector decomposition is used to find alternative representations of   SPD matrices, and to form a Riemannian manifold so that scaling and rotations of SPD matrices are captured by geodesics on this manifold.
The problems of non-unique eigen-decompositions and eigenvalue multiplicities are addressed by finding minimal-length geodesics, which gives rise to a distance and an interpolation method for SPD matrices.
Computational procedures to evaluate the minimal  scaling--rotation deformations and distances are provided for the most useful cases of $2 \times 2$ and $3 \times 3$ SPD matrices.
In the new geometric framework, minimal scaling--rotation curves  interpolate eigenvalues at constant logarithmic rate, and eigenvectors at constant angular rate. In the context of diffusion tensor imaging,  this results in better behavior of
 the trace, determinant and fractional anisotropy of interpolated SPD matrices in typical cases.
\end{abstract}



\pagestyle{myheadings}
\thispagestyle{plain}
\markboth{S. JUNG, A. SCHWARTZMAN AND D. GROISSER}{SCALING--ROTATION OF SPD MATRICES}

\section{Introduction}

The analysis of symmetric positive-definite (SPD) matrices as data objects arises in many contexts. A prominent example is diffusion tensor imaging (DTI), which is a widely-used technique that measures the diffusion of water molecules in a biological object \cite{Basser1994,LeBihan2001,Alexander2005}. The diffusion of water is characterized by a 3D tensor, which is a $3 \times 3$ SPD matrix. The SPD matrices also appear in other contexts of tensor computing \cite{Pennec2006}, tensor-based morphometry \cite{Lepore2008} and as covariance matrices \cite{vemulapalli2015riemannian}.
In recent years statistical analyses of SPD matrices have been received great attention \cite{Zhu2009,Schwartzman2008,Schwartzman2008a,Schwartzman2010,Moakher2011,Yuan2012,Osborne2013}.

The main challenge in the analysis of SPD matrices is that the set of $p \times p$ SPD matrices, $\symp(p)$, is a proper open subset of a real matrix space, so it is not a vector space. This has led researchers to consider alternative geometric frameworks  to handle analytic and statistical tasks for SPD matrices.
The most popular framework is a Riemannian framework, where the set of SPD matrices is endowed with an affine-invariant Riemannian metric \cite{Moakher2005,Pennec2006,Lenglet2006,Fletcher2007}. The Log-Euclidean metric, discussed in \cite{Arsigny2007}, is also widely used, because of its simplicity. \cite{Dryden2009} lists these popular approaches including the Cholesky decomposition-based approach of \cite{Wang2004} and their own approach which they call the Procrustes distance. \cite{Bonnabel2009} proposed a different Riemannian approach for symmetric positive {semidefinite} matrices of fixed rank.

Although these approaches are powerful in generalizing statistics to SPD matrices, they are not easy to interpret in terms of SPD matrix deformations. In particular, in the context of DTI, tensor changes are naturally characterized by changes in diffusion orientation and intensity, but  the above frameworks do not provide such an interpretation.
\cite{Schwartzman2006} proposed a scaling--rotation curve in $\symp(p)$, which is interpretable as rotation of diffusion directions and scaling of the main modes of diffusivity.
In this paper we develop a novel framework to formally characterize scaling--rotation deformations between SPD matrices and introduce a new distance, called here the scaling--rotation distance, defined by the minimum amount of rotation and scaling needed to deform one SPD matrix into another.

To this aim, an alternative representation of  $\symp(p)$, obtained from the decomposition of each SPD matrix into an eigenvalue matrix and eigenvector matrix, is identified as a Riemannian manifold.
This manifold, a generalized cylinder embedded in a higher-dimensional matrix space, is easy to endow with a Riemannian geometry.  A careful analysis is provided to handle the case of equal eigenvalues and, more generally, the non-uniqueness of the eigen-decomposition.
We show that the scaling--rotation curve corresponds to geodesics in the new geometry, and characterize the family of geodesics. A minimal deformation of SPD matrices in terms of the smallest amount of scaling and rotation is then found by a minimal scaling--rotation curve, through a minimal-length geodesic. Sufficient conditions for the uniqueness of minimal curves are given.

The proposed framework not only provides a minimal deformation, but also yields a distance between SPD matrices. This distance function is a semi-metric on $\symp(p)$ and invariant to simultaneous rotation, scaling and inversion of SPD matrices. The invariance to matrix inversion is particularly desirable in analysis of DTI data, where both large and small diffusions are unlikely \cite{Arsigny2007}. While these invariance properties are also found in other frameworks \cite{Moakher2005,Pennec2006,Lenglet2006,Fletcher2007,Arsigny2007}, the proposed distance is directly interpretable in terms of the relative scaling of eigenvalues and rotation angle between eigenvector frames of two SPD matrices.

For $\symp(3)$, other authors \cite{collard2012anisotropy,yang2012feature} have proposed dissimilarity-measures and interpolation schemes based on the same general idea as ours, \emph{i.e.}, separating the scaling and rotation of SPD matrices. Their deformations of SPD matrices can be similar to ours in many cases, thus enjoying similar interpretability.
But while  \cite{collard2012anisotropy,yang2012feature}  mainly focused on the $p =3$ case, our work is more flexible by allowing \emph{unordered} and \emph{equal} eigenvalues. We discuss the importance of this later in Section~\ref{sec:minimalscrotcurve}.


The proposed geometric framework for analysis of SPD matrices is viewed as an important first step to develop statistical tools for SPD matrix data that will inherit the interpretability and the advantageous regular behavior of the scaling--rotation curve. Development of tools similar to those already existing for other geometric framework, such as bi- or tri-linear interpolations \cite{Arsigny2007}, weighted geometric means and spatial smoothing \cite{Moakher2005,Dryden2009,Carmichael2013}, principal geodesic analysis \cite{Fletcher2007}, regression and statistical testing \cite{Zhu2009,Schwartzman2008a, Schwartzman2010,Yuan2012}, will also be needed in the new framework, but we do not address them here.
The proposed framework also has potential future applications beyond diffusion tensor study such as high-dimensional factor models \cite{forni2000generalized} and classification among SPD matrices \cite{Jung2014,vemulapalli2015riemannian}. Algorithms allowing fast computation or approximation of the proposed distance may be needed, but we will leave this  as a subject of future work.
The current paper focuses only on analyzing minimal scaling--rotation curves and the distance defined by them.

The main advantage of the new geometric framework for SPD matrices is that  minimal scaling--rotation curves  interpolate eigenvalues at constant logarithmic rate, and eigenvectors at constant angular rate, with a minimal amount of scaling and rotation.
These are desirable characteristics in fiber-tracking in DTI \cite{Batchelor2005}. Moreover, scaling--rotation curves exhibit regular evolution of determinant, and in typical cases, of fractional anisotropy and mean diffusivity.
Linear interpolation of two SPD matrices by the usual vector operation is known to have a \emph{swelling} effect: the determinants of interpolated SPD matrices are larger than those of the two ends. This is physically unrealistic in DTI \cite{Arsigny2007}. The Riemannian frameworks in \cite{Moakher2005,Pennec2006,Arsigny2007} do not suffer from the {swelling} effect, which was in part the rationale to favor the more sophisticated geometry. However, all of these  exhibit a \emph{fattening} effect: interpolated SPD matrices are more isotropic than the two ends \cite{Chao2009681}. The Riemannian frameworks also produce an unpleasant \emph{shrinking} effect: the trace of interpolated SPD matrices are smaller than those of the two ends \cite{Batchelor2005}. The scaling--rotation framework, on the other hand, does not suffer from the fattening effect and produces a smaller shrinking effect with no shrinking at all in the case of pure rotations.

The rest of the paper is organized as follows.
Scaling--rotation curves are formally defined in Section~\ref{sec:ScalingRotation}.
 Section~\ref{sec:minimalscrotcurve} is devoted to precisely characterizing minimal scaling--rotation curves between two SPD matrices and the distance obtained accordingly. The cylindrical representation of $\symp(p)$ is introduced to handle the non-uniqueness of the eigen-decomposition and repeated eigenvalue cases.
 Section~\ref{sec:computation} provides  details for the computation of the distance and curves for the special but most commonly useful cases of $2\times 2$ and $3 \times 3$ SPD matrices.
In Section~\ref{sec:interpolation}, we highlight the advantageous regular evolution of the scaling--rotation interpolations of SPD matrices. 
 Technical details including proofs of theorems are contained in Appendix.

\section{Scaling--rotation curves in $\symp(p)$}\label{sec:ScalingRotation}


An SPD matrix $M \in \symp(p)$ can be identified with an ellipsoid in $\Real^p$ (ellipse if $p = 2$). In particular, the surface coordinates $x \in \Real^p$ of the ellipsoid corresponding to $M$ satisfy $x'M^{-1}x = 1$. The semi-principal axes of the ellipsoid are given by eigenvector and eigenvalue pairs of $M$.
Fig.~\ref{fig:scarotdeformationsex}  illustrates some SPD matrices in $\symp(3)$ as ellipsoids in $\Real^3$.
Any deformation of the SPD matrix $X$ to another SPD matrix can be achieved by the combination of two operations:
\begin{remunerate}
\item individual scaling of the eigenvalues, or stretching (shrinking) the ellipsoid along principal axes;
\item rotation of the eigenvectors, or rotation of the ellipsoid.
\end{remunerate}

Denote an eigen-decomposition of $X$ by $X = UDU'$, where the columns of $U \in \SO(p)$ consist of orthogonal eigenvectors of $X$, and $D \in \diagp(p)$ is the diagonal matrix of positive eigenvalues that need not be ordered. Here, $\SO(p)$ denotes the set of $p\times p$ real rotation matrices.
To parameterize scaling and rotation, the matrix exponential and logarithm, defined in Appendix~\ref{sec:preliminaries}, are used.
A continuous scaling of the eigenvalues in $D$ at a constant proportionality rate can be described by a curve  $D(t) = \exp(L t) D$ in $\diagp(p)$ for some $L=\mbox{diag}(l_1,\ldots,l_p) \in \diag(p)$, $t \in \Real$, where $\diag(p)$ is the set of all $p\times p$ real diagonal matrices.
Since $\frac{d}{dt}D(t) = LD(t)$, we call $L$ the \emph{scaling velocity}. Each element $l_i$ of $L$ provides the scaling factor for the $i$th coordinate $d_i$ of $D$.
 A rotation of the eigenvectors in the ambient space at a constant ``angular rate'' is described by a curve $U(t) = \exp(At)U$ in $\SO(p)$, where $A \in \asym(p)$, the set of antisymmetric matrices (the Lie algebra of $\SO(p)$). Since $\frac{d}{dt}U(t) = AU(t)$, we call $A$ the \emph{angular velocity}.
Incorporating the scaling and rotation together results in the general scaling--rotation curve (introduced in \cite{Schwartzman2006}),
\begin{equation}\label{eq:sc-rot-curve}
\chi(t) = \chi(t; U,D,A,L) = \exp(At)U D\exp(Lt) U'\exp(A't) \in \symp(p), \quad t \in \Real.
\end{equation}

\begin{figure}[tb!]
 \begin{center}
 \includegraphics[width=0.65\textwidth]{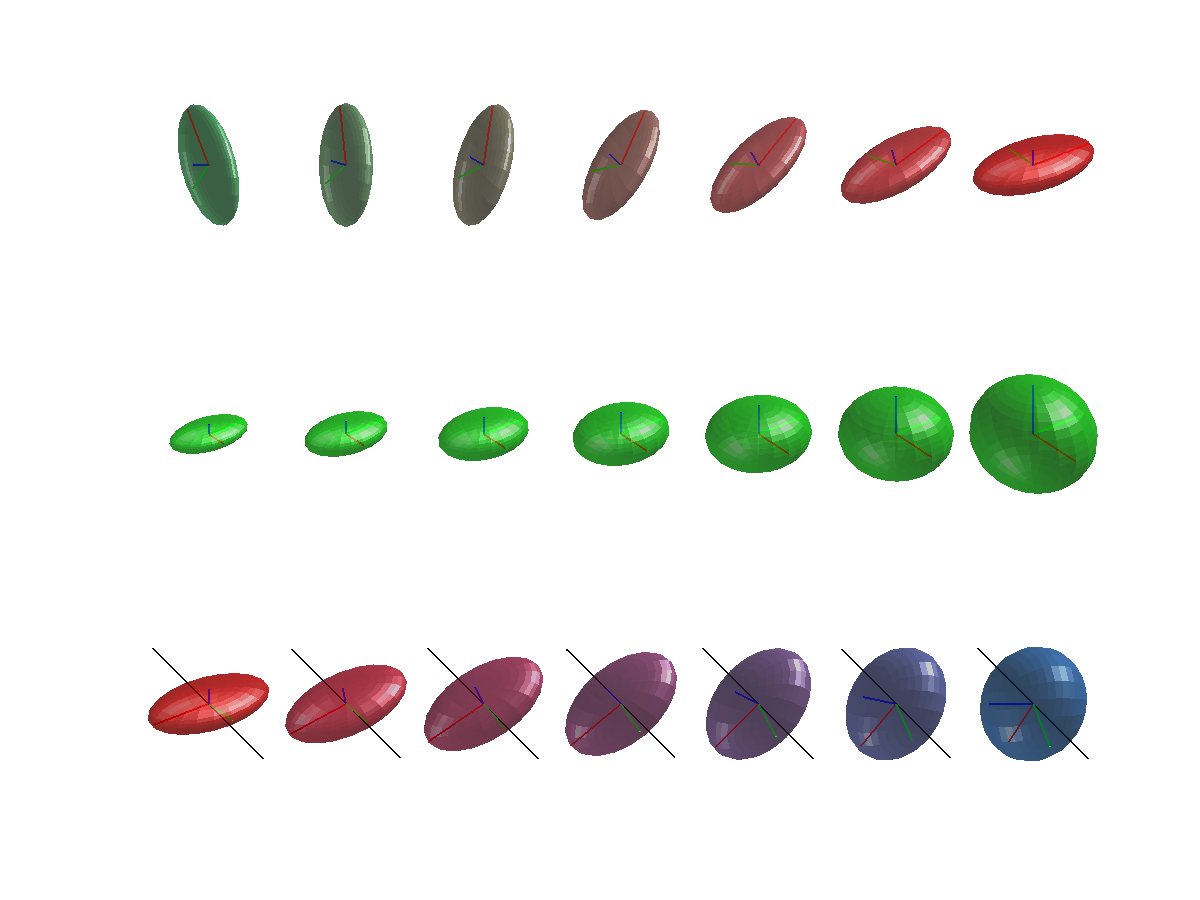}
\end{center}
  \caption{Scaling--rotation curves in $\symp(3)$: (top) pure rotation with rotation axis normal to the screen, (middle) individual scaling along principal axes without any rotation, and (bottom) simultaneous scaling and rotation. The rotation axis is shown as a black line segment. The ellipsoids are colored by the direction of principal axes, to help visualize the effect of rotation.
  \label{fig:scarotdeformationsex}}
\end{figure}

The scaling--rotation curve characterizes deformations of $X = \chi(0) \in \symp(p)$ so that  the ellipsoid corresponding to $X$ is smoothly  rotated, and each principal axis stretched and shrunk, as a function of $t$. For $p = 2, 3$, the matrix $A$ gives the axis and angle of rotation (\emph{cf}. Appendix~\ref{sec:preliminaries}).
%
Fig.~\ref{fig:scarotdeformationsex} illustrates discretized trajectories of scaling--rotation curves in $\symp(3)$, visualized by the corresponding ellipsoids.
These curves in general do not coincide with straight lines or geodesics in other geometric frameworks such as \cite{Wang2004,Moakher2005,Pennec2006,Lenglet2006,Fletcher2007,Arsigny2007,Dryden2009,collard2012anisotropy,yang2012feature}.
In section \ref{sec:minimalscrotcurve}, we introduce a Riemannian metric which reproduces these scaling--rotation curves as images of geodesics.


Given two points $X,Y \in \symp(p)$, we will define the distance between them as the length of a scaling--rotation curve $\chi(t)$ that joins $X$ and $Y$. Thus it is of interest to identify the parameters of the curve $\chi(t)$ that starts at $X= \chi(0)$ and meets $Y = \chi(1)$ at $t=1$.
  From eigen-decompositions of $X$ and $Y$, $X = UDU'$, $Y = V\Lambda V'$, we could equate $\chi(1)$ and $V\Lambda V'$,
and naively solve for eigenvector matrix and eigenvalue matrix separately, leading to
$ A = \log(VU') \in \asym(p)$, $L = \log(D^{-1}\Lambda) \in \diag(p).$
This solution is generally correct, if the eigen-decompositions of X and Y are chosen carefully (see Theorem~\ref{thm:main_result}). The difficulty is that there are many other scaling--rotation curves that also join $X$ and $Y$, due to the non-uniqueness of eigen-decomposition. Thus it is required to consider a \emph{minimal} scaling--rotation curve among all such curves.
%

\section{Minimal scaling--rotation curves  in $\symp(p)$ }\label{sec:minimalscrotcurve}
\subsection{Decomposition of SPD matrices into scaling and rotation components}
An SPD matrix $X$ can be eigen-decomposed into a matrix of eigenvectors $U \in \SO(p)$ and a diagonal matrix $D \in \diagp(p)$ of eigenvalues. In general, there are many pairs $(U,D)$ such that $X = UDU'$. Denote the set of all pairs $(U,D)$ by
$$(\SO \times \diagp)(p) = \SO(p) \times \diagp(p).$$

We use the following notations:
\begin{definition} For all pairs $(U,D) \in (\SO \times \diagp)(p)$ such that $X = UDU'$,
\begin{romannum}
\item An eigen-decomposition $(U,D)$ of $X$ is called an (unobservable) \emph{version} of   $X$ in $(\SO \times \diagp)(p)$;
\item $X$ is the \emph{eigen-composition} of $(U,D)$, defined by a mapping $c: (\SO \times \diagp)(p) \to \symp(p)$, $c(U,D) = UDU' = X$.
\end{romannum}
\end{definition}
The many-to-one mapping $c$ from $(\SO \times \diagp)(p)$ to $\symp(p)$ is surjective. (The symbol $c$ stands for \emph{composition}.)
Fig.~\ref{fig:representations} illustrates the relationship between an SPD matrix and its many versions (eigen-decompositions). While $\symp(p)$ is an open cone, the set $(\SO \times \diagp)(p)$ can be understood as the boundary of a generalized cylinder, \emph{i.e.},  $(\SO \times \diagp)(p)$ forms a shape of cylinder whose cross-section is ``spherical'' ($\SO(p)$) and the centers of the cross section are on the positive orthant of $\Real^p$, \emph{i.e.}, $\diagp(p)$. The set $(\SO \times \diagp)(p)$ is a complete Riemannian manifold, as described below  in Section~\ref{sec:RiemannianFramework}.

\begin{figure}[tb]
\centering
 \includegraphics[width=0.8\textwidth]{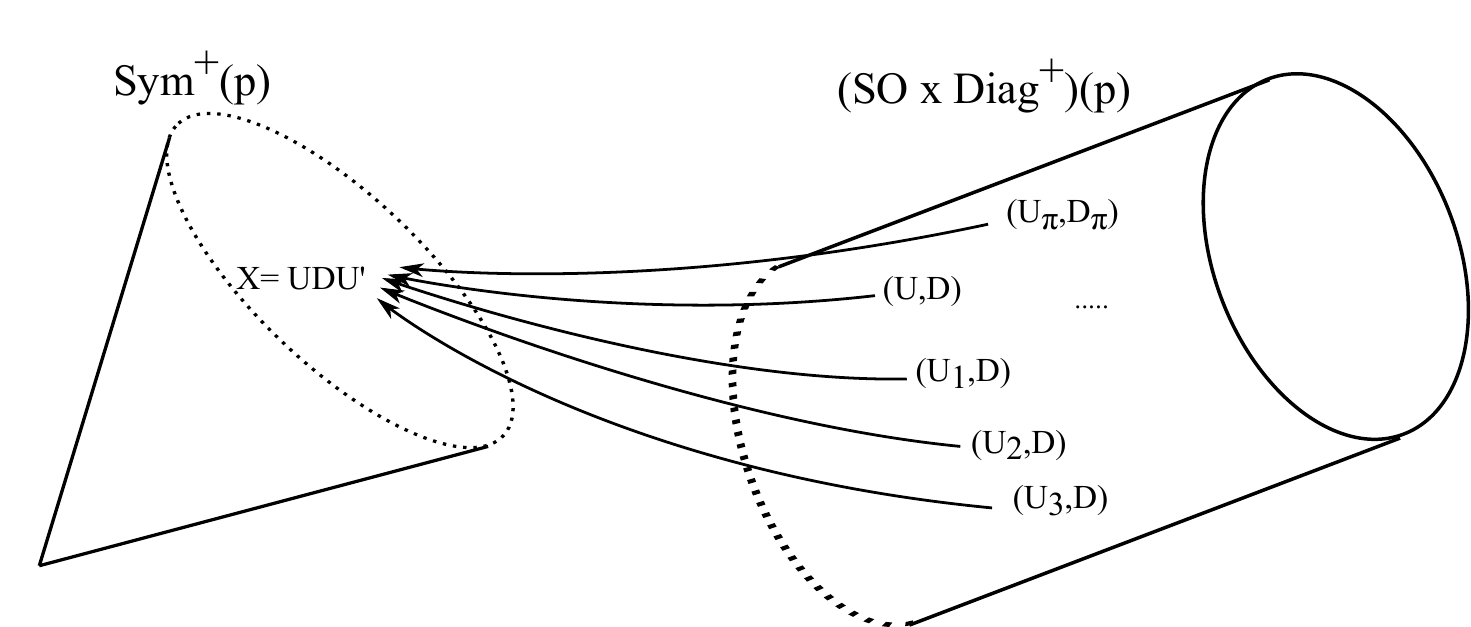}
 \caption{An SPD matrix $X$ and its versions in $(\SO \times \diagp)(p)$. The eigen-composition of $(U,D)$ is depicted as a many-to-one mapping  from $(\SO \times \diagp)(p)$ to $\symp(p)$.\label{fig:representations}}
\end{figure}

Note that considering $(\SO \times \diagp)(p)$ as the set of all possible eigen-decompositions is an important relaxation of the usual ordered eigenvalue assumption.
We will see in the subsequent sections that this is necessary to describe the desired family of deformations. As an example, the scaling--rotation curve depicted at the middle row of Fig.~\ref{fig:scarotdeformationsex} is made possible by allowing \emph{unordered} eigenvalues. Moreover, our manifold $(\SO \times \diagp)(p)$ has no boundaries, which not only allows us to handle \emph{equal} eigenvalues but also makes the applied Riemannian geometry simple.

We first discuss which elements of $(\SO \times \diagp)(p)$ are the versions of any given SPD matrix $X$.

\begin{definition}\label{defn:permutation&signchange}
Let $S_p$ denote the symmetric group, i.e., the group of permutations of the set $\{1,\ldots,p\}$, for $p \ge 2$. A permutation $\pi \in S_p$ is a bijection $\pi: \{1,\ldots,p\} \to \{1,\ldots,p\}$.
Let $\boldsig_p = \{(\epsilon_1,\ldots,\epsilon_p) \in \Real^p : \epsilon_i \in \{ \pm 1\}, 1\le i \le p \}$ and $\boldsig_p^+ = \{ (\epsilon_1,\ldots,\epsilon_p) \in \boldsig_p : \prod_{i=1}^p \epsilon_i = 1\}.$
\begin{romannum}
\item
For a permutation $\pi \in S_p$, its \emph{permutation matrix} is the $p \times p$ matrix $P_\pi^0$ whose entries are all 0 except that in column $i$ the entry $\pi(i)$ equals $1$. Moreover, define
$P_\pi = P_\pi^0$ if $\det(P_\pi^0) = 1$,
$P_\pi = \begin{bmatrix}
-1 &\0v' \\
\0v & I_{p-1} \\
\end{bmatrix}P_\pi^0$ if $\det(P_\pi^0) = -1$.
\item
For $\sigma = (\epsilon_1,\ldots,\epsilon_p) \in \boldsig_p$, its associated \emph{sign-change matrix} is the $p \times p$ diagonal matrix $I_\sigma$ whose $i$th diagonal element is $\epsilon_i$.  If $\sigma \in \boldsig_p^+$, we call $I_\sigma$ an \emph{even sign-change matrix}.
\item For any $D \in \diag(p)$, the \emph{stabilizer subgroup} of $D$ is $G_D = \{R \in \SO(p) : RDR' = D \}$.
\end{romannum}
\end{definition}
%
%
%

For any $\sigma \in \boldsig_p^+$, $\pi \in S_p$, $P_\pi, I_\sigma \in \SO(p)$. The number of different permutations (or sign-changes) is $p!$ (or $2^{p-1}$, respectively). These two types of matrices provide operations for permutation and sign-changes in eigenvalue decomposition. In particular,
for $U \in \SO(p)$, a column-permuted $U$, by a permutation $\pi \in S_p$, is $U P_\pi' \in \SO(p)$, and a sign-changed $U$, by $\sigma \in \boldsig_p^+$, is $U I_\sigma  \in \SO(p)$.
For  $D = \mbox{diag}(d_1,\ldots,d_p)$, define $\pi \cdot D = \mbox{diag} (d_{\pi^{-1}(1)},\ldots,d_{\pi^{-1}(p)}) \in \diag(p)$ as a diagonal matrix whose elements are permuted by $\pi \in S_p$. $D_\pi:= P_\pi D P_\pi'$ is exactly the diagonal matrix $\pi \cdot D$. The same is true if $P_\pi$ is replaced by $I_\sigma P_\pi$, for any $\sigma \in \boldsig_p$. Finally, for any $\sigma \in \boldsig_p^+$, $\pi \in S_p$, there exists $\sigma^0 \in \boldsig_p$ such that $I_{\sigma^0} P_\pi^0 = I_\sigma P_\pi$.


%

\begin{theorem}\label{thm:versions}
Every version of $X = UDU'$ is of the form
$(U^*,D^*) = (URP_\pi', D_\pi),$
for $R \in G_D$, $\pi \in S_p$, and $D_\pi = P_\pi D P_\pi'$.
Moreover, if the eigenvalues of $X$ are all distinct, every $R \in G_D$ is an even sign-change matrix $I_\sigma$, $\sigma \in \boldsig_p^+$.
\end{theorem}

\begin{remark}
{\rm
If the eigenvalues of $X$ are all distinct, there are exactly $p! 2^{p-1}$ eigen-decompositions of $X$. In such a case, all versions of $X$ can be explicitly obtained by application of permutations and sign-changes to any version $(U,D)$ of $X$.
}\end{remark}


\begin{remark} \label{remark:RinRDR=R}
{\rm
If the eigenvalues of  $X$ are not all distinct, there are infinitely many eigen-decompositions of $X$ due to the arbitrary rotation $R$ of eigenvectors. The stabilizer group of $D$, $G_D$, to which $R$ belongs in Theorem~\ref{thm:versions}, does not depend on particular eigenvalues but only on which eigenvalues are equal. More precisely,
for $D = \mbox{diag}(d_1,\ldots,d_p) \in \diagp(p)$, let $\Jc_D$ be the partition of coordinate indices $\{1,\ldots,p\}$ determined by $D$, \emph{i.e.}, for which $i$ and $j$ are in the same block if and only if $d_i = d_j$. A block can consist of non-consecutive numbers.
For a partition $\Jc = \{J_1,\ldots, J_r\}$ with $r$ blocks, let $\{W_1,\ldots, W_r\} = \{\Real^{J_1},\ldots, \Real^{J_r}\}$ denote the corresponding subspaces of $\Real^p$; $x \in \Real^{J_i}$ if and only if the $j$th coordinate of $x$ is $0$ for all $j \notin J_i$.
The stabilizer $G_D$ depends only on the partition $\Jc_D$.
Define  $G_\Jc \subset \SO(p)$ by
\begin{equation}\label{eq:Liesubgroup}
G_\Jc = \{R \in \SO(p) : RW_i = W_i, 1 \le i \le r  \}.
\end{equation}
Then $G_D = G_{\Jc_D}$.
As an illustration, let $D = \mbox{diag}(1,1,2)$. Then $\Jc_D = \{\{1,2\},\{3\}\}$. An example of $R \in G_D$ is a $3 \times 3$ block-diagonal matrix  where the first $2 \times 2$ block is any $R_1 \in \SO(2)$ and the last diagonal element is $r_2 = 1$. Intuitively, $RDR'$ with this choice of $R$ behaves as if the first $2 \times 2$ block of $D$, $D_1$, is arbitrarily rotated. Since $D_1 = I_2$, rotation makes no difference. Another example is given  by setting $R_1 \in \O(2)$ with $\det (R_1) = -1$ and $r_2 = -1$.
%
}\end{remark}

\subsection{A Riemannian framework for scaling and rotation of SPD matrices}\label{sec:RiemannianFramework}

The set of rotation matrices $\SO(p)$ is a $p(p-1)/2$-dimensional smooth Riemannian manifold equipped  with the usual Riemannian inner product for the tangent space \cite[Ch. 18]{Gallier2011}. The set of positive diagonal matrices $\diagp(p)$ is also a $p$-dimensional smooth Riemannian manifold. The set $(\SO \times \diagp)(p)$, being a direct product of two smooth and complete manifolds, is a complete Riemannian manifold  \cite{Small1996,Absil2009}. We state some geometric facts necessary to our discussion.
\begin{lemma}\label{lem:SR1}
\begin{romannum}
\item $(\SO \times \diagp)(p)$ is a differentiable manifold of dimension $p + p(p-1)/2$.
\item $(\SO \times \diagp)(p)$ is the image of $  \asym(p) \times \diag(p)$ under the exponential map
        $ \Exp((A,L)) = (\exp(A),\exp(L))$, $(A,L) \in  \asym(p) \times \diag(p).$
\item The tangent space $\tau(I,I)$ to $(\SO \times \diagp)(p)$ at the identity $(I,I)$ can be naturally identified as a copy of $\asym(p) \times \diag(p)$.
\item The tangent space $\tau(U,D)$ to $(\SO \times \diagp)(p)$ at an arbitrary point $(U,D)$ can be naturally identified as the set $\tau(U,D) = \{(AU,LD): A \in \asym(p), L \in \diag(p)\}$.
\end{romannum}
\end{lemma}

Our choice of Riemannian inner product at $(U,D)$ for two tangent vectors $(A_1U,L_1D)$ and $(A_2U,L_2D)$ is
\begin{align}
       \langle (A_1U,L_1D), (A_2U,L_2D) \rangle_{(U,D)}
       &=  \frac{k}{2}     \langle U'A_1U, U'A_2U \rangle +
       \langle D^{-1}L_1D, D^{-1}L_2D \rangle  \nonumber\\
       & = \frac{k}{2}\tr(A_1A_2') + \tr(L_1L_2),\quad k>0, \label{eq:RiemmanianMetric}
\end{align}
where  $\langle X,Y \rangle$ for $X,Y \in \GL(p)$ denotes the Frobenius inner product $\langle X,Y \rangle = \tr (XY')$.
 \cite{collard2012anisotropy} used a structure similar to (\ref{eq:RiemmanianMetric}), with the scaling factor $k$ being a function of $D$, to motivate their distance function. We use $k = 1$ for all of our illustrations in this paper. The practical effect on using different values of $k$ is discussed in the supplementary material. The practical effect on using different values of $k$ is discussed in Section \ref{sec:discussion}.
For any fixed $k$, we show that this choice of Riemannian inner product leads to interpretable distances with invariance properties (\emph{cf}.  Proposition~\ref{prop:invariance_geoddist} and Theorem~\ref{thm:1alternative}).

 The exponential map from  a tangent space $\tau(U,D)$  to $(\SO \times \diagp)(p)$ is
    $\Exp_{(U,D)}: \tau(U,D) \to (\SO \times \diagp)(p)$,
$$
\Exp_{(U,D)}((AU,LD)) = (U\exp(U'AU), D \exp(D^{-1}LD))
    = ( \exp(A)U, \exp(L)D   ).
$$
The inverse of exponential map is $\Log_{(U,D)}: (\SO \times \diagp)(p) \to \tau(U,D)$,
$$
  \Log_{(U,D)}((V,\Lambda)) = (U\log(U'V), D \log(D^{-1} \Lambda ))
    = (  \log(VU')U ,  \log(\Lambda D^{-1}) D) .
$$
A geodesic in $(\SO \times \diagp)(p)$ starting at $(U,D)$ with initial direction   $(AU,LD) \in \tau(U,D)$ is parameterized as
\begin{equation}\label{eq:geodesicformula}
\gamma(t) = \gamma(t; U,D, A, L) = \Exp_{(U,D)}( (AUt,LDt) ).
\end{equation}

The  inner product (\ref{eq:RiemmanianMetric}) provides the geodesic distance function on $(\SO \times \diagp)(p)$.
Specifically, the squared geodesic distance from  $(U,D)$ to  $(V,\Lambda)$  is
    \begin{align}
d^2\left((U,D),(V,\Lambda)  \right)
 & =  \langle (AU,LD), (AU,LD) \rangle_{(U,D)} \label{eq:geoddist}\\
 & = k d_{{\rm SO}(p)}(U,V)^2 + d_{\Dc}(D_, \Lambda)^2, \quad k>0,  \nonumber
    \end{align}
where $A = \log(VU')$ , $L = \log(\Lambda D^{-1})$, $d_{{\rm SO}(p)}(U_1,U_2)^2 = \frac{1}{2}\norm{\log(U_2U_1')}_F^2$,
$d_{\Dc} (D_1,D_2)^2 = \norm{\log(D_2 D_1^{-1})}_F^2$, and $\| \ \|_F$ is the Frobenius norm.

The geodesic distance (\ref{eq:geoddist}) is a metric, well-defined for any $(U,D)$ and  $(V,\Lambda) \in (\SO \times \diagp)(p)$, and is the length of the minimal geodesic curve $\gamma(t)$ that joins the two points.
Note that for any two points $(U,D)$ and $(V,\Lambda)$, there are infinitely many geodesics that connect the two points, just like there are many ways of wrapping a cylinder with a string.
There is, however, a unique minimal-length geodesic curve that connects  $(U,D)$ and $(V,\Lambda)$ if $VU'$ is not an involution \cite{Moakher2002}. (A rotation matrix $R$ is an \emph{involution} if $R \neq \Id$ and $R^2 = \Id$.)
For $p = 2,3$, $R$ is an involution if it consists of a rotation through angle $\pi$, in which case there exactly two  shortest-length geodesic curves.
If $VU'$ is an involution, then $V$ and $U$ are said to be antipodal in $\SO(p)$, and the matrix logarithm of $VU'$ is not unique (there is no principal logarithm), but as discussed in Appendix~\ref{sec:preliminaries} $\log(VU')$ means any solution $A$ of $\exp(A) = VU'$ whose Frobenius norm is the smallest among all such $A$.
\begin{proposition}\label{prop:invariance_geoddist}
 The geodesic distance (\ref{eq:geoddist}) on $ (\SO \times \diagp)(p)$ is invariant under simultaneous left or right multiplication by orthogonal matrices, permutations and scaling: For any $R_1,R_2 \in \O(p)$, $\pi \in S_p$ and $S \in \diagp(p)$, and for any $(U,D),(V,\Lambda) \in (\SO \times \diagp)(p)$,
  $ d\left((U,D),(V,\Lambda)  \right)
    = d\left((R_1U R_2,S D_\pi),(R_1V R_2 ,S\Lambda_\pi)  \right).$
\end{proposition}

\subsection{Scaling--rotation curves as images of geodesics}

We can give a precise  characterization of  scaling--rotation curves using the Riemannian manifold $(\SO \times \diagp)(p)$.
In particular, any geodesic in $(\SO \times \diagp)(p)$ determines to a scaling--rotation curve in $\symp(p)$.
The geodesic (\ref{eq:geodesicformula}) gives rise to the scaling--rotation curve $\chi(t) = \chi(t; U,D,A,L) \in \symp(p)$  (\ref{eq:sc-rot-curve}), by  the eigen-composition $c \circ \gamma = \chi$.
On the other hand, a scaling--rotation curve $\chi$ corresponds to many geodesics in $(\SO \times \diagp)(p)$.

To characterize the family of geodesics corresponding to a single curve $\chi(t)$, the following notations are used. For a partition $\Jc$ of the set $\{1,\ldots,p\}$, $G_\Jc$ denotes the Lie subgroup of $\SO(p)$ defined in (\ref{eq:Liesubgroup}).
Let $\gf_\Jc$ denote the Lie algebra of $G_\Jc$. Then,
$$\gf_\Jc = \{A \in  \asym(p) : A_{ij} = 0  \mbox{ for } i \not\sim j\}  \subset \asym(p), $$
where $i \not\sim j$ if $i$ and $j$ are   in different blocks of $\Jc$.
For $D \in \diag(p)$, recall from Remark~\ref{remark:RinRDR=R} that $\Jc_D$ is the partition determined by eigenvalues of $D$, $G_D = G_{\Jc_D}$ and define
$\gf_D = \gf_{\Jc_D}$.
For $D, L \in \diag(p)$, let $\Jc_{D,L} $ be the common refinement of $\Jc_D $ and $ \Jc_L$ so that $i$ and $j$ are in the same block of $\Jc_{D,L}$ if and only if $d_i = d_j$ and $l_i = l_j$.
Define $G_{D,L} = G_{\Jc_{D,L}} = G_D \cap G_L$, and let $\gf_{D,L} = \gf_{\Jc_{D,L}} = \gf_D \cap \gf_L$ be the Lie algebra of $G_{D,L}$.
Finally, for $B \in \asym(p)$, let $\ad_B: \asym(p) \to \asym(p)$ be the linear map defined by $\ad_B(C) = [B,C] = BC - CB$.

\begin{theorem}\label{thm:1alternative}
Let $(U,D,A,L)$ be the parameters of a scaling--rotation curve $\chi(t)$ in $\symp(p)$. Let $I$ be a positive-length interval containing $0$.
Then a geodesic $\gamma: I \to (\SO \times \diagp)(p)$ is identified with $\chi$, \emph{i.e.}, $\chi(t)  = c (\gamma(t)),$ for all $t \in I$, if and only if $\gamma(t) = \gamma(t; URP_\pi' ,D_\pi, B, L_\pi)$ for some $\pi \in S_p$, $R \in G_{D,L}$, and $B \in \asym(p)$ satisfying both (i) $\tilde{B} - \tilde{A} \in \gf_{D,L}$, where $\tilde{B} = U'BU$ and $\tilde{A} = U'AU$, and (ii) $(\ad_{\tilde{B}})^j(\tilde{A}) \in \gf_{D,L}$ for all $j \ge 1$.
\end{theorem}

Note that the conjugation $\tilde{A} = U'AU$ expresses the infinitesimal rotation parameter $A$ in the coordinate system determined by $U$.
If $A$ in Theorem~\ref{thm:1alternative} is such that $\tilde{A} \in \gf_{D,L}$, then the conditions (i) and (ii) are equivalent to $\tilde{B} \in \gf_{D,L}$. If $p = 2$ or $3$ and $\tilde{A} \not\in \gf_{D,L}$, then the condition is $\tilde{B} = \tilde{A}$.

It is worth emphasizing a special case where there are only finitely many geodesics corresponding to a scaling--rotation curve $\chi(t)$.

\begin{corollary}\label{cor:them3.8}
Suppose, for some $t$, $\chi(t) = \chi(t; U,D,A,L)$ is an SPD matrix with distinct eigenvalues. Then $\chi$ corresponds to only finitely many $(p!2^{p-1})$ geodesics
$\gamma(t) =  \gamma(t ; U  I_\sigma P_\pi',D_\pi, A, L_\pi)$, where $\pi \in S_p$ and $\sigma \in \boldsig_p^+$.
\end{corollary}
\subsection{Scaling--rotation distance between SPD matrices}

In $(\SO \times \diagp)(p)$, consider the set of all elements whose eigen-composition is $X$:
$$\Ec_X = \{(U,D)\in (\SO \times \diagp)(p) : X = UDU'\}.$$
Since the eigen-composition is a surjective mapping, the collection of these sets $\Ec_X$ partitions the manifold $(\SO \times \diagp)(p)$. The set $\Ec_X = c^{-1}(X)$ is called the fiber over $X$. 
Theorem~\ref{thm:versions} above characterizes all members of $\Ec_X$ for any $X$.

%
%

It is natural to define a distance between  $X$ and $Y$ $\in \symp(p)$ to be the length of the shortest geodesic connecting $\Ec_X$ and $\Ec_Y \subset (\SO \times \diagp)(p)$.
\begin{definition}
For $X,Y \in \symp(p)$, the scaling--rotation distance is defined as
\begin{equation}\label{eq:quotientdistance}
d_{\Sc\Rc} (X,Y)
   := \inf_{ \substack{
     (U,D) \in \Ec_X, \\
    (V,\Lambda) \in \Ec_Y }
   } d( (U,D), (V,\Lambda) ),
\end{equation}
where $d(\cdot,\cdot)$ is the geodesic distance function (\ref{eq:geoddist}).
\end{definition}
%

The geodesic distance $d( (U,D), (V,\Lambda) )$ measures the length of the shortest geodesic segment connecting $(U,D)$ and  $(V,\Lambda)$.
Any geodesic, mapped to $\symp(p)$ by the eigen-composition, is a  scaling--rotation curve connecting $X = UDU'$ and $Y = V\Lambda V'$. In this sense, the scaling--rotation distance $d_{\Sc\Rc}$ measures the minimum amount of smooth deformation from $X$ to $Y$ (or vice versa) only by the rotation of eigenvectors and individual scaling of eigenvalues.

Note that $d_{\Sc\Rc}$ on $\symp(p)$  is well-defined and the infimum is actually achieved, as both $\Ec_X$ and $\Ec_Y$ are non-empty and compact.
It has desirable invariance properties, and is a semi-metric on $\symp(p)$.
\begin{theorem}\label{thm:properties}
For any $X,Y \in \symp(p)$, the scaling--rotation distance $d_{\Sc\Rc}$ is
\begin{romannum}
\item invariant under  matrix inversion, i.e.,  $d_{\Sc\Rc}(X,Y) = d_{\Sc\Rc}(X^{-1},Y^{-1})$,
\item invariant under simultaneous uniform scaling and conjugation by a rotation matrix, i.e., $d_{\Sc\Rc}(X,Y)  = d_{\Sc\Rc}(s RXR' ,s RYR' )$ for any $s > 0$, $R \in \SO(p)$,
\item a semi-metric on $\symp(p)$. That is,
$d_{\Sc\Rc}(X,Y) \ge 0$,
 $d_{\Sc\Rc}(X,Y) = 0$ if and only if $X =Y$, and
$d_{\Sc\Rc}(X,Y) = d_{\Sc\Rc}(Y,X)$.
%
\end{romannum}
\end{theorem}

Although $d_{\Sc\Rc}$ is not a metric on the entire set $\symp(p)$, it is a metric on an important subset of $\symp(p)$. 

\begin{theorem}\label{thm:properties2}
$d_{\Sc\Rc}$ is a metric on the set of SPD matrices whose eigenvalues are all distinct.
\end{theorem}
\subsection{Minimal scaling--rotation curves in $\symp(p)$}
To evaluate the scaling--rotation distance (\ref{eq:quotientdistance}), it is necessary to find a shortest-length geodesic in $(\SO \times \diagp)(p)$ between the fibers $\Ec_X$ and $ \Ec_Y$. There are multiple geodesics connecting two fibers, because each fiber contains at least $p!2^{p-1}$ elements (Theorem~\ref{thm:versions}), as depicted in Fig.~\ref{fig:horizontalgeodesics}.
We think of fibers $\Ec_X$ arranged vertically in  $(\SO \times \diagp)(p)$ with the mapping $c$ (eigen-composition) as downward projection.
It is clear that there exists a geodesic that joins the two fibers with the minimal distance.
We call such a geodesic a \emph{minimal geodesic} for the two fibers $\Ec_X$ and $\Ec_Y$.
A necessary, but generally not sufficient, condition for a geodesic to be minimal for $\Ec_X$ and $\Ec_Y$ is that it is perpendicular to $\Ec_X$ and $\Ec_Y$ at its endpoints.
A pair $((U,D), (V,\Lambda)) \in \Ec_X \times \Ec_Y$ is called a \emph{minimal pair} if $(U,D)$ are $(V,\Lambda)$ are connected by a minimal geodesic.
The  distance $d_{\Sc\Rc} (X,Y)$ is the length of any minimal geodesic segment connecting the fibers $\Ec_X$ and $ \Ec_Y$.

\begin{figure}[tb]
\centering
 \includegraphics[width=0.8\textwidth]{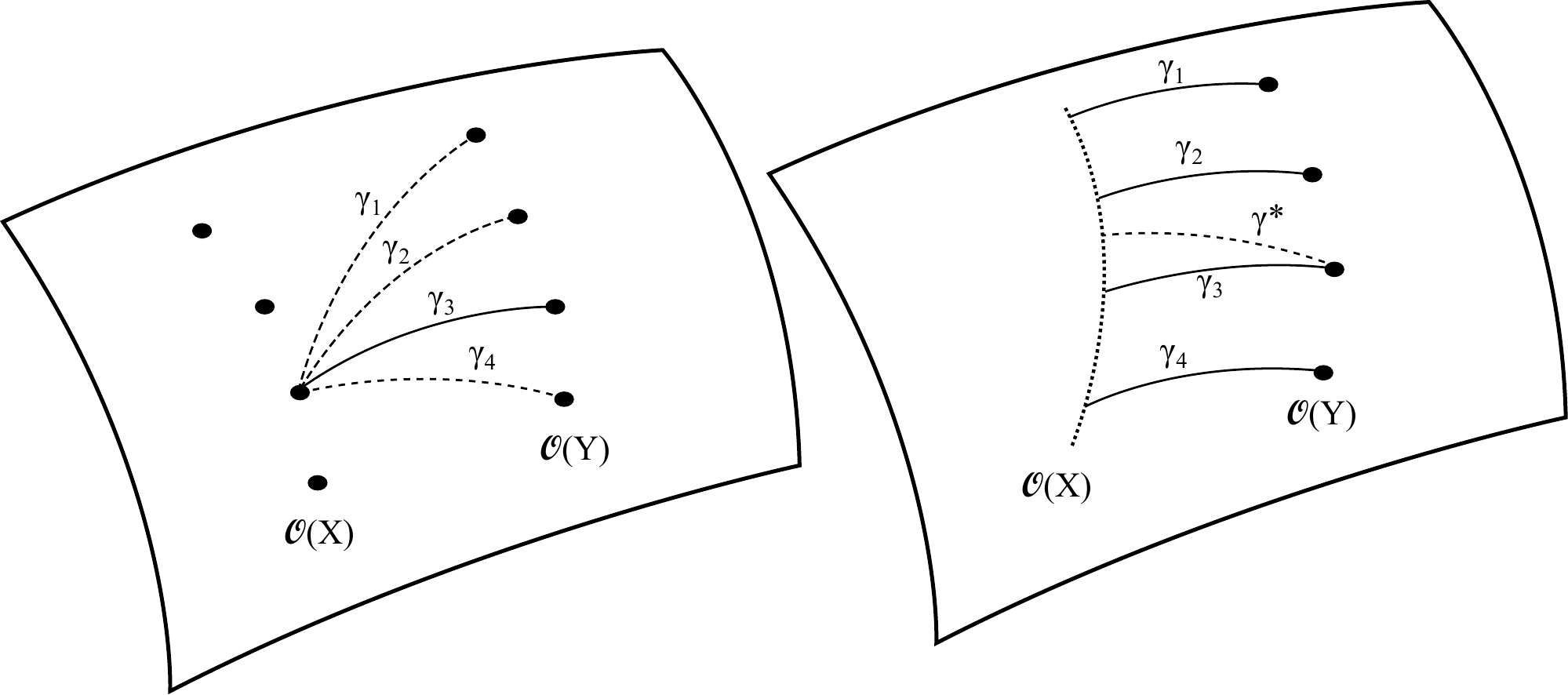}
 \caption{(left) $(\SO \times \diagp)(2)$ is drawn  as a curved manifold. In this picture, the four versions of $X$ (and of $Y$) are displayed vertically. For a fixed version $(U_3,D_3)$ of $X$, there are four geodesics $\gamma_i$ joining $(U_3,D_3)$ and the $i$th version of $Y$. A minimal geodesic ($\gamma_3$ in this figure) has the shortest length among these geodesics.
 (right) The fiber $\Ec_X$ has infinitely many versions, shown as a vertical dotted curve in $(\SO \times \diagp)(2)$. There exist multiple minimal  geodesics $\gamma_i$ with the shortest length, all of which meet the  vertical fiber $\Ec_X$ in the right angle. Here, $\gamma*$ is an example of a non-minimal geodesic, which does not meet $\Ec_X$ orthogonally. \label{fig:horizontalgeodesics}}
\end{figure}

\begin{definition}\label{def:minimalSCROTcurve}
 Let $X, Y \in \symp(p)$.
A scaling--rotation curve $\chi: [0,1] \to \symp(p)$, as defined in (\ref{eq:sc-rot-curve}), with $\chi(0) = X$ and $\chi(1) = Y,$  is called \emph{minimal} if $\chi = c\circ\gamma$ for some minimal geodesic segment $\gamma$ connecting $\Ec_X$ and $\Ec_Y$.
\end{definition}

\begin{theorem}\label{thm:main_result}
Let $X, Y \in \symp(p)$.
Let $((U,D),(V,\Lambda))$ be a minimal pair for $X$ and $Y$, and
let $A = \log(VU'), L = \log(D^{-1} \Lambda)$.
Then the scaling-rotation curve $\chi(t; U,D,A,L)$, $0 \le t \le 1$, is minimal.
\end{theorem}

The above theorem tells us that for any two points $X,Y \in \symp(p)$, a minimal scaling--rotation curve is determined by a minimal pair of $\Ec_X$ and $\Ec_Y$.
Procedures to evaluate the parameters of the minimal rotation--scaling curve and to compute the scaling--rotation distance are provided for the special cases $p = 2,3$ in Section~\ref{sec:computation}.


The minimal scaling--rotation curve  may not be unique.
%
The following theorem gives sufficient conditions for uniqueness. 

\begin{theorem}\label{thm:horizontal_geodesic_and_scarotcurve}
Let $((U,D),(V,\Lambda))$  be a minimal pair for $\Ec_X$ and $\Ec_Y$, and let
$\chi_o(t) = \chi(t; U,D,  \log(VU'), \log(D^{-1}\Lambda))$ be the corresponding minimal scaling--rotation curve.
\begin{enumerate}
\item[(i)] If either all eigenvalues of $D$ are distinct or $\Lambda$ has only one distinct eigenvalue, and if $(V,\Lambda)$ is the unique minimizer of $d((U,D),(V_0,\Lambda_0))$ among all $(V_0,\Lambda_0) \in \Ec_Y$, then all minimal geodesics between $\Ec_X$ and $\Ec_Y$   are mapped by $c$ to the unique $\chi_o(t)$ in $\symp(p)$.
\item[(ii)] If  there exists $(V_1,\Lambda_1) \in \Ec_Y$ such that $(V_1,\Lambda_1) \neq(V,\Lambda)$ and the pair $((U,D),(V_1,\Lambda_1))$ is also minimal, then $\chi_1(t) = \chi(t;   U,D,  \log(V_1U'), \log(D^{-1}\Lambda_1))$ is also minimal and $\chi_1(t) \neq \chi_o(t)$ for some $t$.
\end{enumerate}
\end{theorem}


The following example shows a case with a unique minimal scaling--rotation curve, and two  cases exhibiting non-uniqueness.


\emph{Example}.
{\rm
Consider $X = \mbox{diag}(e,e^{-1})$ and $Y = R_\theta (2X) R_\theta'$, where $R_\theta$ is the $2 \times 2$ rotation matrix
by counterclockwise angle $\theta$. 

\emph{(i)} If $\theta = \pi/3$, then there exists a unique minimal scaling--rotation curve between $X,Y$. This ideal case is depicted in Fig.~\ref{fig:example_paper}, where among the four scaling--rotation curves,  the red curve $\chi_4$  is minimal as indicated by the length of the curves. In the upper right panel, a version $(I,X)$ of $X$, depicted as a diamond, and a version of $Y$ are  joined by the red minimal geodesic segment.

\begin{figure}[tb!]
 \centering
  \includegraphics[width=1\textwidth]{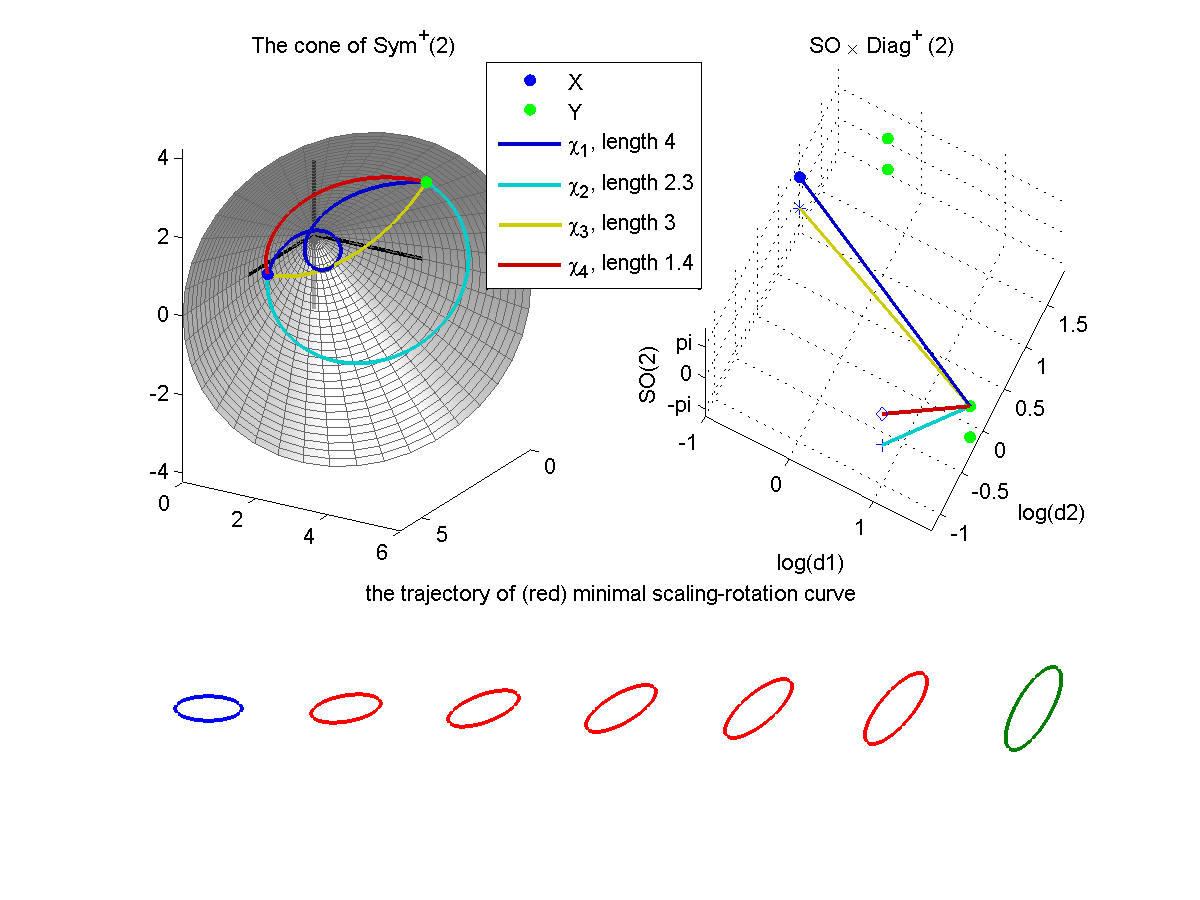}
  \vspace{-0.8in}
  \caption{ Two SPD matrices $X$ (blue) and $Y$ (green) in the cone of $\symp(2)$ (top left), and their four versions in a flattened $(\SO \times \diagp)(2)$ (top right). The eigen-composition of each shortest geodesic connecting versions of $X$ and $Y$ is a scaling--rotation curve in  $\symp(2)$. Different colors represent four different such curves. The red scaling--rotation curve has the shortest geodesic distance in $(\SO \times \diagp)(2)$, and thus is   minimal. Its trajectory is shown as the deformation of ellipses in the bottom panel (from leftmost $X$ to rightmost $Y$). \label{fig:example_paper}}
\end{figure}
\emph{(ii)} Suppose $\theta = \pi/2$. There are two minimal scaling--rotation curves, one by uniform scaling and counterclockwise rotation, the other by the same uniform scaling but by clockwise rotation. 

\emph{(iii)} Let $X = \mbox{diag}(e^{\epsilon/2},e^{-\epsilon/2})$ and $Y = R_\theta X R_\theta'$.
For $0\le\theta \le \pi/2$,
\begin{align*}
    d_{\Sc\Rc}(X,Y)& = \min\left\{\theta, \sqrt{(\frac{\pi}{2}-\theta)^2 + 2\epsilon^2} \right\}
     =      \left\{
                    \begin{array}{ll}
                      \theta, &  {\theta \le \frac{\pi}{4} +  \frac{2\epsilon^2}{\pi}}, \\
                      \sqrt{(\frac{\pi}{2}-\theta)^2 + 2\epsilon^2}, & \hbox{otherwise.}
                    \end{array}
                  \right.
\end{align*}
If the rotation angle is less than 45 degrees or the SPD matrices are highly anisotropic (large $\epsilon$), then the minimal scaling--rotation is a pure rotation (leading to the distance $\theta$). On the other hand, if the matrices are close to be isotropic (eigenvalues $\approx$ 1), the minimal scaling--rotation curve is given by  simultaneous rotation and scaling. An exceptional case arises when $\theta = \frac{\pi}{4} +  \frac{2\epsilon^2}{\pi}<\frac{\pi}{2}$, where both curves are of the same length, and there are two minimal scaling--rotation curves.

}

\section{Computation of the minimal scaling--rotation curve and  scaling--rotation distance}\label{sec:computation}

We provide computation procedures for the scaling--rotation distance $d_{\Sc\Rc}(X,Y)$ for $X,Y \in \symp(2)$ or $\symp(3)$. Theorems~\ref{thm:minimaldistancep=2} and \ref{thm:minimaldistancep=3} below provide the minimal pair(s), based on which the exact formulation of the minimal scaling--rotation curve is evaluated in Theorem~\ref{thm:main_result} above.



\subsection{Scaling--rotation distance for $2 \times 2$ SPD matrices}
Let $(d_1,d_2)$ be the eigenvalues of $X$, $(\lambda_1,\lambda_2)$ the eigenvalues of $Y$.

\begin{theorem}\label{thm:minimaldistancep=2}
Given any $2 \times 2$ SPD matrices $X$ and $Y$, the distance (\ref{eq:quotientdistance}) is computed as follows.
\begin{romannum}
\item If $d_1 \neq d_2$ and $\lambda_1 \neq \lambda_2$, then there are exactly four versions of $X$, denoted by $(U_i, D_i), i=1,\ldots,4$, and for any version $(V,\Lambda)$ of $Y$,
\begin{equation} \label{eq:thm:minimaldistancep=2}
d_{\Sc\Rc}(X,Y) = \min_{i=1,\ldots,4} d((U_i, D_i), (V,\Lambda)  ).
\end{equation}
    These versions are given by the permutation and sign changes.
\item If $d_1 = d_2$, then for any version $(V,\Lambda)$ of $Y$,
     $d_{\Sc\Rc}(X,Y) = d((V, D), (V,\Lambda)  ),$
     regardless of whether the eigenvalues of $Y$ are distinct or not.
\end{romannum}
\end{theorem}
Therefore, the minimizer of $(U_o, D_o)$ of (\ref{eq:thm:minimaldistancep=2}) and $(V,\Lambda)$ are a minimal pair for the case (\emph{i}); $(V, D), (V,\Lambda)$ are a minimal pair for the case (\emph{ii}). 

\subsection{Scaling-rotation distance for $3 \times 3$ SPD matrices}
Let $X,Y \in \symp(3)$.
Let $(d_1,d_2,d_3)$ be the eigenvalues of $X$, $(\lambda_1,\lambda_2,\lambda_3)$ the eigenvalues of $Y$, without any given ordering. 
In order to separately analyze and catalogue all cases of eigenvalue multiplicities in Theorem~\ref{thm:minimaldistancep=3} below, we will use the following details for the case where an eigenvalue of $X$ is of multiplicity 2.


For any version ($U,D$) with $D = \mbox{diag}(d_1,d_1,d_3)$, $d_1=d_2$, all other versions of $X$ are of the form $(UR_1P_\pi', D_\pi)$ for permutation $\pi$ and rotation matrix $R_1 \in G_D$ (Theorem~\ref{thm:versions}).
We can take $R_1 = RI_\sigma$ for some $\sigma \in \boldsig_p^+$ and for some diagonal rotation matrix $R$ with +1 on the lower right hand corner.
 For fixed $(U,D)$, $(V,\Lambda)$, $\sigma \in \boldsig_p^+$ and $\pi \in S_p$, one can find a \textit{minimal rotation} $\hat{R}_{\sigma,\pi}$ satisfying
$$d((U \hat{R}_{\sigma,\pi} I_\sigma P_\pi', D_\pi),(V,\Lambda)) \le d((UR I_\sigma P_\pi', D_\pi),(V,\Lambda)),$$
for all such $R$, as the following lemma states.

\begin{lemma}\label{lem:minimalrotation}
Let $\Gamma = I_\sigma P_\pi' V'U =
 \begin{bmatrix}
               \Gamma_{11} & \Gamma_{12} \\
               \Gamma_{21} & \gamma_{22} \\
             \end{bmatrix}$, where $\Gamma_{11}$ is the first $2 \times 2$ block of $\Gamma$.
 The minimal rotation matrix $\hat{R}_{\sigma,\pi} = \hat{R} $ is given by
$ \hat{R} = \begin{bmatrix}
               E_2E_1' & 0 \\
               0 & 1 \\
             \end{bmatrix},$
             where $E_1\Lambda_\Gamma E_2' $ is the ``semi-singular values'' decomposition of $\Gamma_{11}$. (In semi-singular values decomposition, we require $E_1, E_2 \in \SO(2)$ and that the diagonal entries $\lambda_1$ and $\lambda_2$ of $\Lambda_\Gamma$ satisfy $\lambda_1 \ge |\lambda_2| \ge 0 $.)
\end{lemma}

%
%
%
Each choice of $\sigma$ and $\pi$ produces a {\em minimally rotated version} $(\hat{U}_{\sigma,\pi}, D_\pi) = \linebreak (U \hat{R}_{\sigma,\pi} I_\sigma P'_\pi, D_\pi)$.
To provide a minimal pair as needed in Theorem~\ref{thm:main_result}, a combinatorial problem involving the $3! 2^{3-1} = 24$ choices of $(\sigma, \pi)$ needs to be solved, since the version of $X$ closest to $ (V,\Lambda)$ is found by comparing distances between  $(\hat{U}_{\sigma,\pi}, D_\pi) $ and $(V,\Lambda)$.
Fortunately, there are only six such minimally rotated versions corresponding to six choices of $(\sigma, \pi)$.
%
%
%
%
%
%
In particular, we  need only
$ \pi_1 : (1,2,3) \to (1,2,3)$,
$ \pi_2 : (1,2,3) \to (3,1,2)$,
$ \pi_3 : (1,2,3) \to (1,3,2)$,
and ${\sigma_1}  =(1,1,1)$, ${\sigma_2} = (-1,1,-1),$ and $\hat{R}_{{\sigma_j}, \pi_i}$ can be found for each
$({\sigma_j}, \pi_i)$, $i=1,2,3,j=1,2$.
The other pairs of permutations and sign-changes do not need to be considered because
each of them will produce one of the six minimally rotated versions, with the same distance from $(V,\Lambda)$.


\begin{theorem} \label{thm:minimaldistancep=3}
Given any $3 \times 3$ SPD matrices $X$ and $Y$, the distance (\ref{eq:quotientdistance}) is computed as follows.
\begin{enumerate}
\item[(i)] If the eigenvalues of $X$ (and also of $Y$) are all distinct, then  there are exactly twenty four versions of $X$, denoted by $(U_i, D_i), i=1,\ldots,24$, and for any version $(V,\Lambda)$ of $Y$,
    $d_{\Sc\Rc}(X,Y) = \min_{i=1,\ldots,24} d((U_i, D_i), (V,\Lambda)  ).$
\item[(ii)] If $d_1 = d_2 \neq d_3$ and $\{\lambda_1,\lambda_2,\lambda_3\}$ are distinct, then for any version $(V,\Lambda)$ of $Y$ and a version $(U,D)$ of $X$ satisfying $D = \mbox{diag}(d_1,d_1,d_3)$,
    $$d_{\Sc\Rc}(X,Y) = \min_{i=1,2,3, j = 1,2} d((\hat{U}_{\sigma_j,\pi_i}, D_{\pi_i}) , (V,\Lambda)  ),$$
    where $(\hat{U}_{\sigma_j,\pi_i}, D_{\pi_i})$, $i=1,2,3, j = 1,2$ are the six minimally rotated versions.
\item[(iii)] If $d_1 = d_2 \neq d_3$ and $\lambda_1 = \lambda_2 \neq \lambda_3$, choose $D = \mbox{diag}(d_1,d_2,d_3)$ and $\Lambda = \mbox{diag}(\lambda_1,\lambda_2,\lambda_3)$. For any versions $(U,D)$, $(V,\Lambda)$ of $X$ and $Y$,
    $$d_{\Sc\Rc}(X,Y) = \min_{i=1,2,3, j = 1,2} d(( U R_{\theta_{ij}} I_{\sigma_i}P_{\pi_j}', D_{\pi_j}), (VR_{\phi_{ij}},\Lambda)  ),$$
   where $R_\theta = \exp([a]_\times)$, $a = (0,0,\theta)'$ (cf. Appendix~\ref{sec:preliminaries}),
   and $({\theta_{ij}},\phi_{ij})$ simultaneously maximizes
    $G(\theta, \phi) = \tr (U R_\theta I_{\sigma_i}P_{\pi_j}' R_\phi' V').$
\item[(iv)] If $d_1 = d_2 = d_3$, then for any version $(V,\Lambda)$ of $Y$,
     $d_{\Sc\Rc}(X,Y)= d((V, D), (V,\Lambda)  ),$
     regardless of whether the eigenvalues of $Y$ are distinct or not.
\end{enumerate}
\end{theorem}

The minimizer $({\theta_{ij}},\phi_{ij})$ of $G(\theta,\phi)$ in Theorem~\ref{thm:minimaldistancep=3}(iii) 
is found by a numerical method. Specifically, given the $m$th iterates $\theta^{(m)}, \phi^{(m)}$, the $(m+1)$th iterate $\theta^{(m+1)}$ is the solution $\theta$ in Lemma~\ref{lem:minimalrotation}, treating $VR_{\phi^{(m)}}$ as $V$. We then find $\phi^{(m+1)}$ similarly by using  Lemma~\ref{lem:minimalrotation}, with the role of $U$ and $V$ switched. In our experiments, convergence to the unique maximum was fast and reached by only a few iterations.

\section{Scaling--rotation interpolation of SPD matrices}\label{sec:interpolation}
For $X,Y \in \symp(p)$, \emph{ a scaling--rotation interpolation from $X$ to $Y$} is defined as any minimal scaling--rotation curve $f_{SR}(t) := \chi_o(t)$, $t \in [0,1]$, such that $f_{SR}(0) = X$, $f_{SR}(1) = Y$.
By definition, every scaling--rotation curve $\chi(t; U,D,A,L)$, and hence every scaling--rotation interpolation, has a log-constant scaling velocity $L$ and constant angular velocity $A$. The scalar $\tr(L)$ gives the (constant) speed at which log-determinant evolves: $\log(\det\chi(t)) = \log(\det(D)) + \tr(L) t$. Analogously, we view the scalar quantity $\norm{A}_F/\sqrt{2}$ as a constant {\em speed of rotation}, and for all $t\geq 0$ we define the \emph{amount of rotation} applied from time 0 to time $t$ to be $\theta_t := t \norm{A}_F/\sqrt{2}$. For a minimal pair $((U,D),(V,\Lambda))$ of $X$ and $Y$, and the corresponding scaling--rotation interpolation $f_{SR}$, we have
\begin{equation}\label{eq:logdet}
\log(\det f_{SR}(t) ) = (1-t) \log(\det(X))  + t \log(\det(Y)),
\end{equation}
and we define the \emph{ amount of rotation applied by $f_{SR}$ from $X$ to $Y$} to be
$\theta := \norm{\log(VU')}_F/\sqrt{2}$.
For $p = 2,3$, $\theta$ is equal to the angle of rotation.

\subsection{An application to diffusion tensor computing}\label{sec:applications}
This work provides an interpretative geometric framework in analysis of diffusion tensor magnetic resonance images \cite{LeBihan2001}, where diffusion tensors are given by $3\times 3$ SPD matrices. Interpolation of  tensors is important for fiber tracking, registration and spatial normalization of diffusion tensor images \cite{Batchelor2005,Chao2009681}.
The scaling--rotation curve can be understood as a deformation path from one diffusion tensor to another, and is nicely interpreted as  scaling of diffusion intensities and rotation of diffusion directions. This advantage in interpretation has not been found in popular geometric frameworks such as \cite{Pennec2006,Fletcher2007,Arsigny2007,Dryden2009,Bonnabel2009}. The approaches in \cite{Batchelor2005,Chao2009681,collard2012anisotropy,yang2012feature} also explicitly use rotation of directions and many scaling--rotation curves are very similar to the deformation paths given in \cite{collard2012anisotropy,yang2012feature}. We defer the discussion on the difference between our framework and those in \cite{collard2012anisotropy,yang2012feature} to Section~\ref{sec:discussion}.

As an example, consider interpolating from $X = \mbox{diag}(15,2,1)$ to $Y$, whose eigenvalues are $(100,2,1)$ and whose principal axes are different from those of $X$.
The first row of Fig.~\ref{fig:linearinterpolation} presents the corresponding evolution ellipsoids by the scaling--rotation interpolation $f_{SR}$. This evolution is consistent with human perception when deforming $X$ to $Y$. As shown in the left two bottom panels of Fig.~\ref{fig:linearinterpolation}, the interpolation  exhibits the constant angular rate of rotation, and log-constant rate of change of determinant.

\begin{figure}[tb!]
 \centering
  \includegraphics[width=1\textwidth]{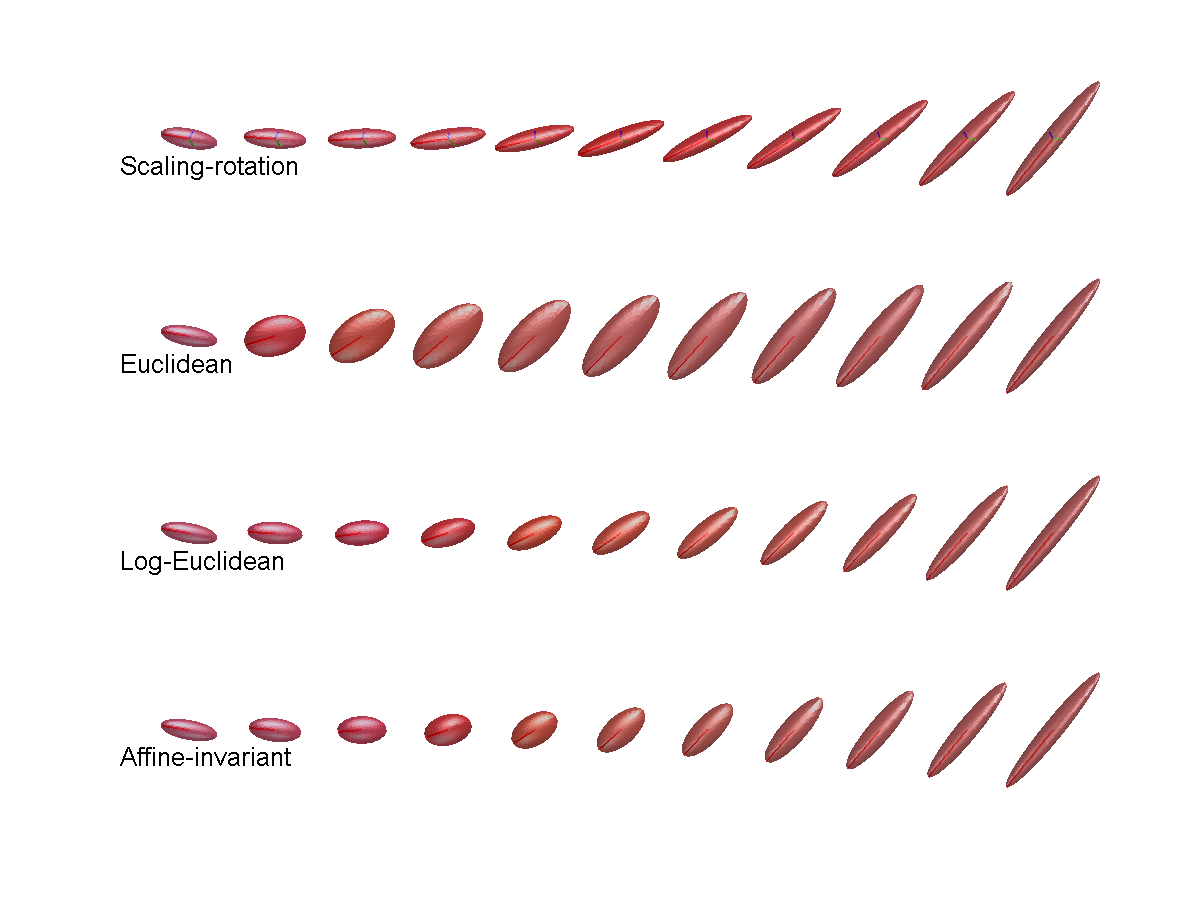}\\
 \vskip -0.3in
  \includegraphics[width=1\textwidth]{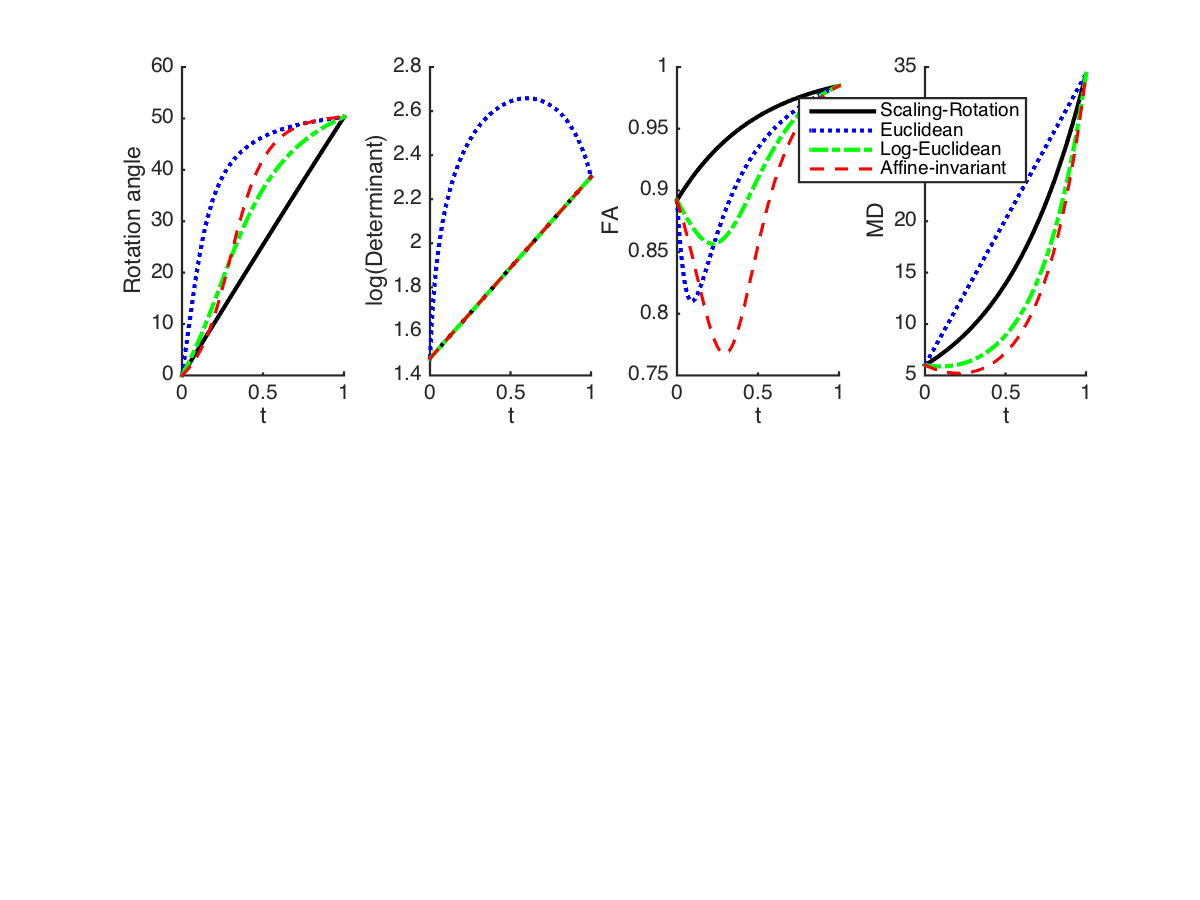}
  \vspace{-2in}
  \caption{(Top) Interpolations of two $ 3\times 3$ SPD matrices. Row 1: Scaling--rotation interpolation by the minimal scaling--rotation curve. Row 2: (Euclidean) linear interpolation on coefficients. Row 3: Log-Euclidean geodesic interpolation. Row 4: Affine-invariant Riemannian interpolation.  The pointy shape of ellipsoids on both ends is well-preserved in the scaling--rotation interpolation. \label{fig:linearinterpolation}
  (Bottom) Evolution of rotation angle, determinant, FA and MD for these four interpolations. Only the scaling--rotation interpolation provides a monotone pattern.}
\end{figure}

By way of comparison, the Euclidean interpolation in row 2 is defined by $f_E(t) = (1-t) X + t Y$. The log-Euclidean and affine-invariant Riemannian interpolation in rows 3 and 4  are defined by $f_{LE}(t) = \exp(  (1-t) \log(X)+ t \log(Y))$ and $f_{AI}(t) = X^{\half} \exp( t \log( X^{-\half} Y  X^{-\half})) X^{\half}$, respectively; see \cite{Arsigny2007}. For these interpolations, we define the rotation angle at time $t$ by the angle of swing from the major axis at time 0 to that at time $t$. These rotation angles are not in general linear in $t$, as the bottom left panel illustrates.
The log-Euclidean $f_{LE}$ and affine-invariant interpolations $f_{AI}$ are log-linear in determinant, and in fact (\ref{eq:logdet}) holds exactly for $f_{LE}$ and $f_{AI}$.
On the other hand, the Euclidean interpolation is known to suffer from the \emph{swelling effect}: $\det(f_E(t))  > \max(\det(X),\det(Y))$, for some $t \in [0,1]$ \cite{Arsigny2007}. This is shown in the bottom second panel of Fig.~\ref{fig:linearinterpolation} for the same example. The other interpolations, $f_{SR},f_{LE}$ and $f_{AI}$ do not suffer from the swelling effect.

Minimal scaling–rotation curves not only provide regular evolution of rotation
angles and determinant, but also minimize the combined amount of scaling
and rotation, as in Definition \ref{def:minimalSCROTcurve}. This results in a particularly desirable property: in many examples, the fractional anisotropy (FA) and mean diffusivity (MD) evolve monotonically. FA measures a degree of anisotropy that is zero if all eigenvalues are equal, and approaches 1 if one eigenvalue is held constant and the other two approach zero; see \cite{LeBihan2001}. MD is the average of eigenvalues $\mbox{MD}(X) = \tr(X)/3$.

In the example of Fig.~\ref{fig:linearinterpolation},  FA$(f_{SR}(t))$  increases monotonically.
In contrast, other interpolations of the highly anisotropic $X$ and $Y$  become less anisotropic.
This phenomenon may be called a \emph{fattening effect}: interpolated SPD matrices are more isotropic than the two ends \cite{Chao2009681}.
Moreover, log-Euclidean and affine-invariant Riemannian interpolations can suffer from a \emph{shrinking effect}: the MD of interpolated SPD matrices are smaller than those of the two ends \cite{Batchelor2005}, as shown in the bottom right panel. In this example, the scaling--rotation interpolation does not suffer from  fattening and shrinking effects. These adverse effects are less severe in $f_{SR}$ than in $f_{LE}$ or $f_{AI}$ in most typical examples, as shown in the online supplementary material.

The advantageous regular evolution results from the rotational part of the interpolation. To see this, consider, as in \cite{Batchelor2005}, a case where the interpolation by $f_{SR}$  consists only of rotation (a precise example is shown in Figure 1 in the online supplementary material).  The scaling-rotation interpolation preserves the determinant, FA and MD, while the other modes of interpolations exhibit irregular behavior in some of the measurements.
On the other hand, when $f_{SR}(t)$ is composed of pure scaling, then $f_{SR}(t) =  f_{LE}(t) = f_{AI}(t)$ for all $t$, and there is no guarantee that MD or FA grow monotonically for neither curve. The equality of these three curves in this special case is a consequence of the geometric scaling of eigenvalues in the scaling--rotation curve (\ref{eq:sc-rot-curve}), which in turn is a consequence of our use of  the Riemannian inner product (\ref{eq:RiemmanianMetric}).

In summary, while the three most popular methods suffer from swelling, fattening, or shrinking effects, the scaling--rotation interpolation  provides good regular evolution of all three summary statistics, and solely provides constant angular rate of rotation. More examples illustrating these effects in various scenarios are given in the online supplementary material.

\subsection{Comparison to other rotation--scaling schemes}\label{sec:discussion}
Geometric frameworks for $3 \times 3$ SPD matrices that decouple rotation from scaling have also been developed by \cite{collard2012anisotropy,yang2012feature}. However, our framework differs in two major ways. First, we allow unordered and equal eigenvalues in any dimension, while \cite{collard2012anisotropy,yang2012feature} considered only dimension 3 and only the case of distinct, ordered eigenvalues. In our framework, every scaling--rotation curve corresponds to geodesics in a smooth manifold, which is not possible if eigenvalues are ordered. This leads to a more flexible family of interpolations than those of \cite{collard2012anisotropy,yang2012feature}, as we illustrate in the online supplementary material.

Another difference lies in the choice of  the metric for $\SO(3)$, and the weight $k$ in (\ref{eq:RiemmanianMetric}).
While we use geodesic distance and interpolation determined by the standard Riemannian metric on $\SO(3)$, \cite{collard2012anisotropy} used a chordal distance and extrinsic interpolation.
As a consequence, the interpolation in \cite{collard2012anisotropy} is close to, but not equal to, a special case of minimal scaling--rotation curves, in particular when $k$ in (\ref{eq:RiemmanianMetric}) is small. An example illustrating this effect of $k$ is given in Section 2 of the online supplementary material.

\Appendix
\appendix
\section{Parameterization of scaling and rotation}\label{sec:preliminaries}%

The matrix exponential of a square matrix $Y$ is  $\exp({{Y}}) = \sum_{j=0}^{\infty}\frac{{{Y}}^j}{ {j!}}.$
For a square matrix $X$, if there exists a unique matrix $Y$ of smallest norm such that $X = \exp(Y)$, then we call $Y$ the principal logarithm of $X$, denoted $\log(X)$.


%
%

The exponential map from $\diag(p)$ to $\diagp(p)$, defined by the matrix exponential, is bijective. Moreover, the element-wise exponential and logarithm for the diagonal elements give the   matrix exponential and logarithm for $\diag(p)$ and $\diagp(p)$ \cite[Ch.18]{Gallier2011}.


Rotation matrices can be parameterized by antisymmetric matrices since the exponential map from $\asym(p)$  to $\SO(p)$ is onto.
Contrary to the $\diagp(p)$ case, the matrix exponential is not one-to-one.
The principal logarithm is defined on the set $\{R \in \SO(p) : R$ is not an involution$\}$, a dense open subset of $\SO(p)$.
For completeness, when there exists no principal logarithm of $R$, we use the notation $\log(R)$ to denote any solution $A$ of $\exp(A) = R$ satisfying that $\norm{A}_F$ is the smallest among all such choices of $A$.

This parameterization gives a physical interpretation of rotations. Specifically, in the case of $p = 3$,
%
%
a rotation matrix $R = \exp(A) \in\SO(3)$ can be understood as a linear operator rotating a vector in the real 3-space around an ``axis'' $a = (a_1,a_2,a_3)'$ by angle $\theta = \norm{a}_2$ (in radians), where $A \in \asym(3)$ is the cross-product matrix of $a$ defined by
$$A = [a]_\times =
\begin{bmatrix}
  0 & -a_3 & a_2 \\
  a_3 & 0 & -a_1 \\
  -a_2 & a_1 & 0 \\
\end{bmatrix} \in \asym(3).$$
Explicit formulas for the matrix exponential and logarithm are given in the following.
%
%
\begin{lemma}[\cite{Moakher2002,Gallier2011}]\label{lem:rotation3}
\begin{enumerate}
 \item[(i)] (Rodrigues' formula) Any $A \in \asym(3)$ equals $[a]_\times$ for some axis $a = \theta \tilde{a} \in \Real^3$, $\norm{\tilde{a}} = 1$.  An explicit formula for the matrix exponential of $A$ is  $\exp(A) = \Id+ [\tilde{a}]_\times \sin(\theta) +  [\tilde{a}]_\times^2 (1-\cos(\theta)).$

 \item[(ii)] For any $R \in \SO(3)$, there exists  $\theta \in [0,\pi]$ satisfying
   $ 2\cos(\theta) = \tr(R) - 1,  \norm{\log (R)}_F = \sqrt{2} |\theta|.$ If $\theta<\pi$, then $R$ has the principal logarithm $\log(R) = \frac{\theta}{2\sin(\theta)}(R - R')$ if $\theta \in (0,\pi)$, $0$ if $\theta = 0$.
\end{enumerate}
\end{lemma}
If there exists no principal logarithm of $R$, i.e., if $\theta = \pi$, then we take $\log(R)$ to be  either  of the two elements $A \in \asym(3)$ satisfying $\exp(A) = R$ and $\norm{A}_F = \sqrt{2} \pi$.

\section{Additional lemmas and proofs}

\emph{Proof of Theorem \ref{thm:versions}}.
For any $(U^*,D^*) \in (\SO \times \diagp)(p)$, there exists $V \in \SO(p)$ , $L \in \diagp(p)$ such that $U^* = UV$ and $D^* = DL$. Therefore, a version of $(U,D)$ can be written as $(UV,DL)$ satisfying
$UVDLV'U' = UDU',$
or equivalently,
\begin{equation}\label{eq:version_requirements}
VDLV' = D.
\end{equation}
The set of eigenvalues of $VDLV'$ is $\{d_il_i: i = 1,\ldots,p\}$, which should be the same as the eigenvalues of $D$. That is, $d_il_i = d_j$ for some $j$. In other words, for some permutation $\pi$, $DL = D_\pi$. There are at most $p!$ possible ways to achieve this. 

Observe that there exists $R \in \SO(p)$ such that $RP_\pi' = V$, for any $V$ and $\pi$. (\ref{eq:version_requirements}) is then $RP_\pi' D_\pi P_\pi R' = D$, which becomes $RDR' = D$ since $P_\pi' D_\pi P_\pi = D$.

The last statement of Theorem~\ref{thm:versions} can be seen from noting that $R$ and $D$ commute, so the eigenvector matrix of $R$ is $I$, with eigenvalues $\{e^{i\theta_j},e^{-i\theta_j},1, j=1,\ldots,\lfloor p/2\rfloor \}$ \cite[Corollary 2.5.11]{Horn2012}. However, all possible values of $\theta_j$ are either $0$ or $\pi$ because $R$ must be a real matrix. Therefore $R$ is an even sign-change matrix.
$\endproof$

\emph{Proof of Proposition \ref{prop:invariance_geoddist}}.
Since $S$ commutes with other diagonal matrices,
\begin{align*}
    d^2 & \left((R_1UR_2,SD_\pi),(R_1 V R_2,S\Lambda_\pi)  \right) \\
       & = \frac{k}{2}\norm{\log(R_1 U R_2 R_2' V' R_1')}^2_F + \tr(\log^2(S\Lambda_\pi D_\pi^{-1}S^{-1}))\\
      & = \frac{k}{2}\norm{\log( U V')}^2_F  + \tr(\log^2( \Lambda  D^{-1} )     = d^2\left((U,D),(V,\Lambda)\right). \ \endproof
\end{align*}

\emph{Proof of Theorem~\ref{thm:1alternative}}.
We use the following lemmas.

\begin{lemma}\label{lem:Tbad_finite} For any $D,L \in \diag(p)$, $T_{bad} = \{ t \in \Real : \Jc_{D\exp(Lt)} \neq \Jc_{D,L}\} $ has at most $p(p-1)/2$ elements.
\end{lemma}
\begin{proof}
Let $D = \mbox{diag}(d_1,\ldots,d_p)$ and $L = \mbox{diag}(l_1,\ldots,l_p)$.
For $1\le i < j \le p $, if $i \sim_{D,L} j$, \emph{i.e.}, $d_i = d_j$ and $l_i = l_j$, then $i \sim_{D\exp(Lt)} j$ (or $d_i\exp(l_it) = d_j \exp(l_jt)$) for all $t$ and the indices $i$ and $j$ are in a same block of $ \Jc_{D,L}$ and $\Jc_{D\exp(Lt)}$ for all $t$. If $i \not\sim_{D,L} j$, then $i \sim_{D\exp(Lt)} j$ for at most one $t$. 
 The result follows from the fact that there are only $p(p-1)/2$ pairs of $i$ and $j$.
\end{proof}

\begin{lemma}\label{lem:ABequivalence}
Let $I \subset \Real$ be a positive-length interval containing 0. Let $G \subset \SO(p)$ be a Lie subgroup of $\SO(p)$, and let $\gf \subset \so(p)$ denote the Lie algebra of $G$.
Then for any $A,B \in \so(p)$,
\begin{align}\label{eq:1}
B - A \in \gf,\ (\ad_{{B}})^j({A}) \in \gf \ \mbox{ for all }\ j \ge 1
\end{align}
if and only if
there exists a $C^\infty$ map $g: I \to G$ such that
\begin{equation}\label{eq:2}
\exp(t{B}) = \exp(t{A})g(t) \ \mbox{ for all }\ t \in I.
\end{equation}
\end{lemma}
\begin{proof}
Throughout the proof, the ``prime'' symbol denotes derivative, not transpose.
Suppose first (\ref{eq:2}) holds. For $t \in I$, define
$X(t) = g(t)^{-1}g'(t)$ and $Y(t) = g(t)^{-1} A g(t)$.
Since $g'(t) \in T_{g(t)} G$, $X(t) = g(t)^{-1}g'(t) \in T_{\Id} G - \gf$. Thus, $X$ is a $C^\infty$ map $I \to \gf$, and $Y$ is a $C^\infty$ map $I \to \so(p)$.
Differentiating (\ref{eq:2}) gives
\begin{align*}
\exp(t{B}) B &= \exp(t{A})(A g(t) + g'(t)) = \exp(t{A})g(t)(Y(t) + X(t) \\
             &= \exp(t{B})(Y(t) + X(t).
\end{align*}
Therefore,
\begin{equation}\label{cc}
X(t) + Y(t) = B  \ \mbox{ for all }\ t \in I.
\end{equation}
A simple computation leads that
\begin{equation}\label{15}
Y'(t) = [Y(t),X(t)] = -[X(t),Y(t)].
\end{equation}
From (\ref{cc}), we get
\begin{equation}\label{s1}
X'(t) = - Y'(t) = [X(t), Y(t)] = [X(t), B - X(t)] = [X(t), B],
\end{equation}
and consequently
\begin{equation}\label{s2}
X'(t) = - \ad_B(X(t)).
\end{equation}

Equation (\ref{s2}) is a constant-coefficient linear differential equation for a function $X : I \to \so(p)$. The general solution is therefore
\begin{equation}\label{xt1}
X(t) = \exp( -t \ad_B) X(0).
\end{equation}
From general Lie group theory, we have, for $W \in \gf$,
$\exp(\ad_W) = \mbox{Ad}_{\exp(W)}$, where $\mbox{Ad}_h : \gf \to \gf$ is conjugation by $h$. Thus $X(t)$ can also be written as
\begin{equation}\label{xt2}
X(t) = \exp(-tB) X(0) \exp(tB).
\end{equation}
It is easy to check that (\ref{xt2}) solves the ordinary differential equation (\ref{s2}) (or equivalently, (\ref{s1})).
Now, since $g(0) = \Id$, $Y(0) = A$. Thus (\ref{cc}) implies that $X(0) = B-A$.
Since $X(t)$ lies in the vector space $\gf$, so do the derivatives $X^{(j)}(0)$ of all orders $j \ge 0$. So $B - A = X(0) \in \gf$, $X^{(j)}(0) \in \gf$ for all $j \ge 1$.
So
\begin{align*}
B - A &= X(0) \in \gf,\\
(\ad_{{B}})^j({A}) &= -(\ad_{{B}})^j({B-A}) = \pm X^{(j)}(0) \in \gf \ \mbox{ for all }\ j \ge 1,
\end{align*}
leading to the conditions (\ref{eq:1}).

Next, suppose (\ref{eq:1}) holds.
Define $C = B-A$. Then
$(\ad_{{B}})^j({C}) = -(\ad_{{B}})^j({A}) \in \gf$ for all $j\ge 1 $.
Define $X(t) = \exp( -t \ad_B)C$; then for all $t$
$$X(t) = C + \sum_{j=1}^\infty \frac{1}{j!}(-t \ad_B)^j (C) \in \gf.$$
Note that $X$ is the unique solution of the initial value problem
$X'(t) = -[B,X(t)],\ X(0) = C.$
From (\ref{xt1})-(\ref{xt2}), $X(t)$ can be written as
$$X(t) = \exp(-tB) C \exp(tB) = B - \exp(-tB) C \exp(-tB).$$

It is a known fact that for a compact Lie group $G$ and with Lie algebra $\gf$, and any smooth $X_1: \Real \to \gf$, the initial value problem
$
g(t)^{-1} g'(t) = X_1(t),\ g(0) = \mbox{identity element of} \ G$,
has a unique solution, that the solution is smooth, and that the maximal time-domain of the solution is all of $\Real$. Since $G \subset \SO(p)$ is compact and $X$ is smooth, we have a smooth solution
$g: \Real \to G$ of the  initial value problem
$$g(t)^{-1} g'(t) = X(t),\ g(0) = \Id.$$

Define $Y_1(t) = g(t)^{-1}Ag(t)$. Then, as computed earlier in (\ref{15}), $Y_1'(t) = [Y_1(t),X(t)].$
Define $Y_2(t) = \exp(-tB)A \exp(tB) = B - X(t)$. Then $Y_2'(t) = - X'(t) = [B, X(t)] = [X(t) + Y_2(t), X(t) ] = [Y_2(t),X(t)]$.
Note also that $Y_1(0) = A = Y_2(0)$. Thus
$Y_1,Y_2$ satisfy the same linear initial-value problem for a function $Y: \Real \to \so(p)$, hence are identically equal.
Therefore $X(t) + Y_1(t) = X(t) + Y_2(t) = B$ for all $t$.

Define $h(t) = \exp(tA) g(t) \exp(-tB)$. Then
\begin{align*}
h'(t) &= \exp(tA)(A g(t) + g'(t) - g(t)B) \exp(-tB) \\
      &= \exp(tA) g(t) (Y_1(t) + X(t) -B) \exp(-tB) \equiv 0,
\end{align*}
so $h(t) = h(0) = \Id$ for all $t \in \Real$. Therefore the function $g$ satisfies (\ref{eq:2}).
\end{proof}

We now provide a proof of Theorem~\ref{thm:1alternative}. Let $(V,\Lambda) \in (\SO \times \diagp)(p)$, $B \in \asym(p)$, $N \in \diag(p)$, and $\gamma = \gamma(V,\Lambda, B,N) : I \to (\SO \times \diagp)(p)$.

Suppose first that  $\chi(t)  = c (\gamma(t)),$ for all $t \in I$.
Theorem~\ref{thm:versions} indicates that for all $t \in I$, there exist
$R(t) \in G_{D\exp(Lt)}$ and $\pi_t \in S_p$ such that
\begin{align}
\exp(tB)V       &= \exp(tA) U R(t) P_{\pi_t}', \label{eq:new1}\\
\exp(tN)\Lambda &= \exp(tL_{\pi_t}) D_{\pi_t}. \label{eq:new2}
\end{align}
Setting $t = 0$, we have
\begin{equation}\label{eq:new3}
V =  UR(0)P_{\pi_0}' ,\ \Lambda = D_{\pi_0}.
\end{equation}
From (\ref{eq:new2}) and (\ref{eq:new3}), we get
$tN + \log D_{\pi_0} = t L_{\pi_t} + \log D_{\pi_t}$.
Thus for $i = 1,\ldots, p$ and all $t \in I \backslash \{0\}$,
\begin{equation}\label{eq:new4}
    n_i =  l_{t,i} + \frac{c_{t,i}}{t},
\end{equation}
where $n_i, l_{t,i}$ and $c_{t,i}$ are the $i$th diagonal entries of $N$, $L_{\pi_t}$ and $\log(D_{\pi_t} D^{-1}_{\pi_0})$, respectively. As $t, \pi_t$ and $i$ range over their possible values, there are only finitely many values of $l_{t,i}$ and $c_{t,i}$. Therefore, unless $c_{t,i} = 0$ for all $i$ and $t \neq 0$, the right hand side of (\ref{eq:new4}) is unbounded as $t \to 0$, contradicting the constancy of $n_i$. Thus,
\begin{equation}\label{eq:new5}
    D_{\pi_t} = D_{\pi_0},\  L_{\pi_t} = N = L_{\pi_0} \ \mbox{for all}\ t \in I.
\end{equation}
Next, let $T_{bad} = \{ t \in I : \Jc_{D\exp(Lt)} \neq \Jc_{D,L}\} $, a finite set (\emph{cf.} Lemma~\ref{lem:Tbad_finite}). On $I\backslash T_{bad}$, the partition $\Kc$ determined by $\exp(tN)\Lambda = \exp(tL_{\pi_0})D_{\pi_0}$ is constant, as is the partition $\Jc_{D,L}$ determined by $D\exp(Lt)$. Thus we can replace $\pi_t$, for all $t \in I \backslash T_{bad}$, by a constant permutation $\pi_\infty$ that carries $\Jc_{D,L}$ to $\Kc$, without affecting the truth of (\ref{eq:new1})-(\ref{eq:new2}). With this replacement applied to (\ref{eq:new1}), we have
\begin{equation}\label{eq:new6}
    \exp(tB)V =\exp(tB) U R(0) P_{\pi_0}' = \exp(tA) U R(t) P_{\pi_\infty}',
\end{equation}
for   $t \in I \backslash T_{bad}$.
Letting $\tilde{A} = U'AU$, $\tilde{B} = U'AU$ and $P = P'_{\pi_\infty} P_{\pi_0}$, we rewrite   (\ref{eq:new6}) as
\begin{align}\label{eq:gt}
\exp(t\tilde{B}) = \exp(t\tilde{A}) g(t),
\end{align}
for all $ t \in I \backslash T_{bad},$  where $ g(t) = R(t) P R(0)'$.
Thus
\begin{equation}\label{eq:new7}
    R(t) = \exp(-t\tilde{A}) \exp(t\tilde{B}) R(0) P',  \  t \in I \backslash T_{bad}.
\end{equation}
The right hand side of (\ref{eq:new7}) is continuous on $I$. Hence the left hand side of (\ref{eq:new7}) continuously extends to each $t_{bad} \in  T_{bad}$, replacing  $R(t_{bad})$ with $\lim_{t \to t_{bad}} R(t)$. With this replacement for $R(t)$, (\ref{eq:gt}) is true for all $t \in I$, and the new functions $t \mapsto R(t)$ and $t \mapsto g(t)$ are continuous on all of $I$. Evaluating (\ref{eq:gt}) at $t = 0$, we get $\Id = g(0) = R(0) P R(0)'$, implying that $P = \Id$.

Note that $R(t) \in G_{\Jc_{D,L}}$ for all $t \in I \backslash T_{bad}$. Since $G_{\Jc_{D,L}}$ is a closed subset of $\SO(p)$, continuity implies that $R(t) \in G_{\Jc_{D,L}}$ for all $t \in T_{bad}$ as well. From (\ref{eq:new7}), it then follows that  $t \mapsto R(t)$ is a $C^\infty$ function $I \to G_{\Jc_{D,L}}$. Therefore there exists a $C^\infty$ map $g: I \to G_{\Jc_{D,L}}$ satisfying (\ref{eq:gt}) for all $t \in I$.
Then Lemma~\ref{lem:ABequivalence} shows that the asserted conditions for $B$ are necessary.

For the sufficiency, Lemma~\ref{lem:ABequivalence} shows that there exists a $C^\infty$ function $g: I \to G_{\Jc_{D,L}}$ satisfying (\ref{eq:gt}). From this it is easy to check that, for $R$ and $\pi$ as in the assumption, the eigen-composition of $\gamma(URP_\pi',D_\pi,B, L_\pi)$ is $\chi$.
$\endproof$

\emph{Proof of Corollary \ref{cor:them3.8}}.
Let $t_0 \in I$ be the such that eigenvalues of $\chi(t_0)$ are all distinct. Then
$\Jc_{{D,L}} = \cap_{t \in I} \Jc_{{D\exp(Lt)}} \subset \Jc_{{D\exp(Lt_0)}}$. Since $\Jc_{{D\exp(Lt_0)}}$ has only singleton blocks, so does $\Jc_{{D,L}}$.
Then $\gf_{D,L}$ only consists of $\0v$, and conditions (i)-(ii) in Theorem~\ref{thm:1alternative} are simplified to $B = A$. Finally Theorem~\ref{thm:versions} shows that any $R \in G_{D,L}$ is an even sign-change matrix.  
$\endproof$

\emph{Theorem~\ref{thm:properties}} is easily obtained, and we omit the proof.

\emph{Proof of Theorem~\ref{thm:properties2}}.
By Theorem~\ref{thm:properties}, it is enough to   show the triangle inequality. Fix a version of $Z$, say $(U,D)$. Since all eigenvalues of $D$ are distinct,  there exists a unique minimizer $(V_X,\Lambda_X) \in \Ec_X$ such that $d((V_X,\Lambda_X), (U,D)  ) \le d((V,\Lambda), (U,D)  )$  for all $(V,\Lambda) \in \Ec_X$. Likewise, denote $(V_Y,\Lambda_Y) \in \Ec_Y$ the unique minimizer with respect to $(U,D)$.
Then $d_{\Sc\Rc}(X,Z)+d_{\Sc\Rc}(Z,Y) = d((V_X,\Lambda_X),(U,D)) + d((V_Y,\Lambda_Y),(U,D))$. Since $d$ is a metric, by the triangle inequality for $((\SO\times \diagp)(p), d)$,
$d((V_X,\Lambda_X),(U,D)) + d((V_Y,\Lambda_Y),(U,D)) \ge d((V_X,\Lambda_X),(V_Y,\Lambda_Y))$. The proof is concluded by noting that $d_{\Sc\Rc}(X,Y)  \le d((V_X,\Lambda_X),(V_Y,\Lambda_Y))$.
$\endproof$

\emph{Proof of Theorem~\ref{thm:horizontal_geodesic_and_scarotcurve}}.
The following lemma is  used in the proof.
\begin{lemma}\label{thm:rotation-general}
If $((U,D),(V,\Lambda))$ is minimal for $\Ec_X$ and $\Ec_Y$, then for any  $R \in G_{D,\Lambda}$ and $\pi \in S_p$,  the shortest-length geodesics connecting
\begin{equation} \label{eq:rotation-general}
(U R P_\pi' , D_\pi)\ \mbox{   and   }\ (V R P_\pi' , \Lambda_\pi)
\end{equation}
are also minimal. (These minimal pairs are said to be \emph{equivalent} to each other.)
\end{lemma}

\emph{Proof of Lemma~\ref{thm:rotation-general}.}
The result is obtained by two facts. $(U R P_\pi' , D_\pi)$ is a version of $X = UDU'$ and  $(V R P_\pi', \Lambda_\pi)$ is a version of $Y = V\Lambda V'$; By the invariance of $d$ (Proposition~\ref{prop:invariance_geoddist}), we have  $
d((U R P_\pi', D_\pi), (V R P_\pi', \Lambda_\pi))= d((U, D), (V, \Lambda))$. $\endproof$

To prove (i) of Theorem~\ref{thm:horizontal_geodesic_and_scarotcurve}, suppose $((U_1,D_1),(V_1,\Lambda_1))$ is another minimal pair. Then there exist $R \in G_D$, $\pi \in S_p$ such that
$U_1 = URP_{\pi}'$, $D_1 = D_{\pi}$. By Proposition~\ref{prop:invariance_geoddist},
\begin{align}
 d((U_1,D_1),(V_1,\Lambda_1)) &= d((URP_{\pi}',D_{\pi}),(V_1,\Lambda_1)) \nonumber \\
 &= d((UR,D),(V_1 P_{\pi^{-1}}' ,\Lambda_{1,\pi^{-1}})), \ \mbox{ where }\ \Lambda_{1,\pi^{-1}}  = (\Lambda_1)_{\pi^{-1}}, \nonumber \\
 & =d((U,D),(V_1 P_{\pi^{-1}}' R,\Lambda_{1,\pi^{-1}})). \nonumber
\end{align}
The conditions of $D$ and $\Lambda$ lead that $G_D \subset G_{\Lambda_\pi}$ for all $\pi \in S_p$, which in turn leads $R\in G_{\Lambda_{1,\pi^{-1}}}$.
 Since $(V_1 P_{\pi^{-1}}' R,\Lambda_{1,\pi^{-1}}) \in \Ec_Y$, the uniqueness assumption gives $V_1 = V R P_{\pi}$ and $\Lambda_1 = \Lambda_{\pi}$. Thus by Lemma~\ref{thm:rotation-general}, all minimal pairs are equivalent to each other.
Moreover, the scaling--rotation curve corresponding to the shortest-length geodesic between the minimal pair $(U_1,D_1)$ and $(V_1,\Lambda_1)$ is the same as $\chi_o(t)$ (Theorem~\ref{thm:1alternative}). Thus $\chi_o$ is unique.
%

For (ii), let $\gamma_0(t) = \gamma(t; U,D,A,L)$ and $\gamma_1(t) = \gamma(t; U,D,B,L_1)$ be the two minimal geodesics, where $A = \log(VU'), L = \log(D^{-1}\Lambda), B = \log(V_1U')$ and $L_1 = \log(D^{-1}\Lambda_1)$. The length-minimizing property of minimal geodesics implies that both $\frac{d}{dt}\gamma_0(t) |_{t=0}$ and $\frac{d}{dt}\gamma_1(t) |_{t=0}$ are perpendicular to $T_{(U,D)} \Ec_X$. Note that
\begin{align*}
 T_{(U,D)}\Ec_X = \{(UC, 0) \in \sym(p) \times \diag(p) : C \in \gf_D\}.
\end{align*}
Since left and right translations by $(U,D)$ and $(R, \Id)$ are  isometry of $(\SO \times \diagp)(p)$, we have
$
  (T_{(U,D)}\Ec_X)^{\perp} = U(\gf_D)^\perp  \oplus \diag(p),
$
  where
$ U(\gf_D)^\perp  = \{UW : W \in (\gf_D)^\perp  \subset T_\Id \SO(p) = \asym(p)  \}$.

Let $\tilde{A} = U'AU$, $\tilde{B} = U'BU$.
Since $\frac{d}{dt}\gamma_0(t) |_{t=0} = (U\tilde{A},DL)$ and $\frac{d}{dt}\gamma_1(t) |_{t=0} = (U\tilde{B},DL_1)$, it follows that $\tilde{A}, \tilde{B} \in (\gf_D)^\perp$ and thus $\tilde{B} - \tilde{A} \in (\gf_D)^\perp$.

Let $\chi_0 = c \circ\gamma_0$, $\chi_1 = c \circ \gamma_1$, and assume $\chi_0 = \chi_1$. By the necessary condition (i) in Theorem~\ref{thm:1alternative}, we have $\tilde{B} - \tilde{A} \in \gf_D$. Hence $\tilde{B} - \tilde{A} = \0v$, implying $B=A$ and $V_1 = V$. Then ``$\chi_0 = \chi_1$'' implies $L_1 = L$ as well, a contradiction. $\endproof$

\emph{Proof of Theorem~\ref{thm:minimaldistancep=2}}.
(i) By definition (\ref{eq:quotientdistance}), $d_{\Sc\Rc}(X,Y) = \min_{k,j} d((U_k,D_k),(V_j,\Lambda_j))$, for $k,j = 1,2,3,4$. Suppose, without loss of generality, $(V,\Lambda) = (V_1,\Lambda_1)$. For any choice of $j = 1,2,3,4$, there exist  $\pi \in S_p$ and $\sigma \in \boldsig_p^+$ such that
$(V_jI_\sigma P_\pi, (\Lambda_j)_{\pi}) = (V_1, \Lambda_1)$. Moreover, for any $k$, one can choose some $i$ so that
$(U_kI_\sigma P_\pi, (D_k)_{\pi}) = (U_i, D_i)$. Therefore, with a help of Proposition~\ref{prop:invariance_geoddist}, for any $k,j$, there exist $\pi, I_\sigma$, and $i$ satisfying
$$d((U_k,D_k),(V_j,\Lambda_j))
  = d((U_kI_\sigma P_\pi, (D_k)_{\pi}),(V_jI_\sigma P_\pi, (\Lambda_j)_{\pi}))
  = d( (U_i, D_i) , (V_1, \Lambda_1)  ).$$
  Thus it is enough to fix a version of $Y$ and compare the distances given by four versions of $X$.

(ii) It is clear from the proof of (i) and by Proposition~\ref{prop:invariance_geoddist} that we can fix a version of $Y$ first. Since the eigenvalues of $D$ are identical to, say, $d_1$, $(U,d_1 I_2)$ is a version of $X$ for any $U \in \SO(2)$. Thus choosing $U = V$ leads to the smallest distance between $U,V \in \SO(2)$.
$\endproof$

\emph{Proof of Lemma~\ref{lem:minimalrotation}}.
Note that a matrix $R$ that rotates the first two columns of $U$ when post-multiplied is
$R = \begin{bmatrix}
               R_{11} & 0 \\
               0 & 1 \\
      \end{bmatrix}$, $R_{11} \in \SO(2)$.
Using Lemma~\ref{lem:rotation3}(ii), we have
\begin{align*}
 \argmin_{R} &\ d((UR I_\sigma P_\pi', D_\pi),(V,\Lambda))  = \argmin_{R} \norm{ \log(UR  I_\sigma P_\pi' V')}_F \\
     & = \argmin_{R} \norm{ \log(UR  I_\sigma P_\pi' V')}_F
      = \argmax_{R} \tr{( UR  I_\sigma P_\pi' V')}\\
     & = \argmax_{R} \tr{(I_\sigma P_\pi' V' UR )}
      = \argmax_{R_{11}} \tr{(
           \begin{bmatrix}
                         \Gamma_{11} & \Gamma_{12} \\
                         \Gamma_{21} & \gamma_{22} \\
                       \end{bmatrix}
             \begin{bmatrix}
                                    R_{11} &  0 \\
                                   0 & 1 \\
                                  \end{bmatrix}  )} \\
      & = \argmax_{R_{11}} \tr (\Gamma_{11} R_{11}) + \gamma_{22}.
\end{align*}
Since $R_{11} \in SO(2)$, the singular values of $R_{11}$ are unity. The result is obtained by an application of the fact from \cite{Mirsky1975} that for any square matrices $A$ and $B$ with vectors of singular values $\sigma_A$ and $\sigma_B$ in non-increasing order,
$|\tr(A'B)| \le \sigma_A'\sigma_B$.
$\endproof$

\emph{Proof of Theorem~\ref{thm:minimaldistancep=3}.}
A proof of (i),(ii) and (iv) can be obtained by a simple extension of the proof of Theorem~\ref{thm:minimaldistancep=2}.
For (iii), note that all versions of $X$ and $Y$ are $(UR_\theta I_{\sigma_k}P_{\pi_l}', D_{\pi_l})$ and
$(VR_\phi I_{\sigma_a}P_{\pi_b}', \Lambda_{\pi_b})$. Following the lines of the proof of Theorem~\ref{thm:minimaldistancep=2}(i), by choosing $I_{\sigma_a} = P_{\pi_b} = I$, it is enough to compare $(UR_\theta I_{\sigma_i}P_{\pi_j}', D_{\pi_j})$ and
$(VR_\phi, \Lambda)$. Moreover, the presence of rotation matrix $R_\theta$ allows us to  restrict the choice of $I_\sigma$ and $\pi$ to only six pairs. For a fixed $(i,j)$, $(i = 1,2,3, j =1,2)$,
\begin{align}
\min_{\theta,\phi} d(( U R_{\theta} I_{\sigma_i}P_{\pi_j}', D_{\pi_j}, (VR_{\phi},\Lambda)  )
 & = \min_{\theta,\phi} \norm{ \log( U R_\theta I_{\sigma_i}P_{\pi_j}' R_\phi' V'  )  }_F  \nonumber \\
 & = \max_{\theta,\phi} \tr (U R_\theta I_{\sigma_i}P_{\pi_j}' R_\phi' V'), \label{eq:simulminimizer}
\end{align}
by Lemma~\ref{lem:rotation3}(ii). $\endproof$

%

%

\bibliographystyle{siam}
\bibliography{DTI}

\newpage
\begin{center}
\textbf{\large Supplemental Materials: Scaling-rotation distance and interpolation of symmetric positive-definite matrices}
\end{center}
\setcounter{equation}{0}
\setcounter{figure}{0}
\setcounter{table}{0}
\setcounter{page}{1}
\makeatletter
\renewcommand{\theequation}{S\arabic{equation}}
\renewcommand{\thefigure}{S\arabic{figure}}
\section{Examples of SPD matrix interpolations}

In connection with Section 5 of the main article, we provide visual examples of the scaling--rotation interpolation $f_{SR}$ between $3 \times 3 $ SPD matrices, compared with other interpolations:   Euclidean ($f_E$), log-Euclidean ($f_{LE}$), affine-invariant Riemannian ($f_{AI}$) and the interpolation used in \cite{collard2012anisotropy},  which we will refer to as the \emph{scaling--quaternion} interpolation ($f_{SQ}$).
\cite{collard2012anisotropy} considered only the set $$\symp_*(3) = \{(U,D) \in (\SO \times \diag^+) (3):  d_1 > d_2> d_3, D = \mbox{diag}(d_1,d_2,d_3)\}.$$

The purpose of these additional examples is to illustrate the advantageous regular evolution of rotation angle, determinant, fractional anisotropy and trace along the interpolation path.

We show  interpolations from $X$ to $Y$, for five different pairs $(X,Y)$.
\begin{enumerate}
\item Pure rotation in $f_{SR}$. Fig.~\ref{1}.
\item Pure scaling in $f_{SR}$. Fig.~\ref{2}.
\item A moderate mix of rotation and scaling. Fig.~\ref{3}.
\item A moderate mix of rotation and scaling, but with $\mbox{trace}(X) = \mbox{trace}(Y)$ (so, $\mbox{MD}(X) = \mbox{MD}(Y)$). Fig.~\ref{4}.
\item Departure from isotropy. Fig.~\ref{5}.
\end{enumerate}

In all figures, the top five rows show various interpolations of the given $ 3\times 3$ SPD matrices: Row 1--$f_{SR}$, row 2--$f_E$, row 3--$f_{LE}$, row 4--$f_{AI}$ and row 5--$f_{SQ}$. The ellipsoids are colored by the direction of the first principal axis (red: left--right, green: up--down, blue: in--out), which varies smoothly with $t$ for $f_{SR}$ and $f_{SQ}$.
For the other interpolations in rows 2--4, the first principal axis always corresponds to the largest eigenvalue, and may not vary smoothly (\emph{cf}. Fig.~2).
The bottom panel shows the evolution of rotation angle, determinant, fractional anisotropy (FA) and mean diffusivity (MD) for the five modes of interpolations. MD measures overall diffusion intensity in a tensor $X = V\Lambda V'$ as the average of the eigenvalues  $\mbox{MD}(X) = \bar{\lambda} = \tr (\Lambda) / 3  = (\lambda_1 + \lambda_2 + \lambda_3)/3$. FA measures a degree of anisotropy of the tensor $X$, and is defined by
$$\mbox{FA}(X) = \sqrt{\frac{3}{2}} \frac{\sqrt{(\lambda_1 - \bar{\lambda})^2 + (\lambda_2 - \bar{\lambda})^2 + (\lambda_3 - \bar{\lambda})^2}}{\sqrt{\lambda_1^2 + \lambda_2^2 + \lambda_3^2}}.$$
The rotation angle $\theta_t$ at time $t$ is measured by $\theta_t = \cos^{-1}\left(\frac{1}{2}(\mbox{trace}(U(t) U(0)') - 1)\right)$ for $f_{SR}$ and $f_{SQ}$ (where the eigenvector frame $U(t)$ is given explicitly), and by $\theta_t = \cos^{-1}(| u_1(t)'u_1(0)| )$, where $u_1(t)$ is the eigenvector of $f_E(t)$ (or $f_{LE}(t), f_{AI}(t)$) corresponding to the largest eigenvalue, for $f_E(t) $ (or $f_{LE}(t), f_{AI}(t)$, respectively).
In all the examples in this section, we have chosen the scaling factor $k = 1$ in the scaling--rotation distance and interpolation  (see eq. (3.4) in the main article);  in Section 2, we illustrate the effect of changing $k$.
The coordinate axes in the examples are chosen so that the axis of rotation  $\av = (-0.5272,    -0.6871,     0.5)'$ is normal to the screen.

  \textbf{Case 1}: Pure rotation.
  \begin{align*}
  X &= \mbox{diag}(15,5,1), \\
   Y &= R(\frac{\pi}{3}\av)  \mbox{diag}(15,5,1)    R(\frac{\pi}{3} \av)',
  \end{align*}
where $R(\theta\av)$ is the $3\times 3$ rotation matrix with axis $\av$ and rotation angle $\theta$ (in radians); see Appendix Lemma A.1 (i).

\begin{figure}[tb!]
 \centering
  \includegraphics[width=.8\textwidth]{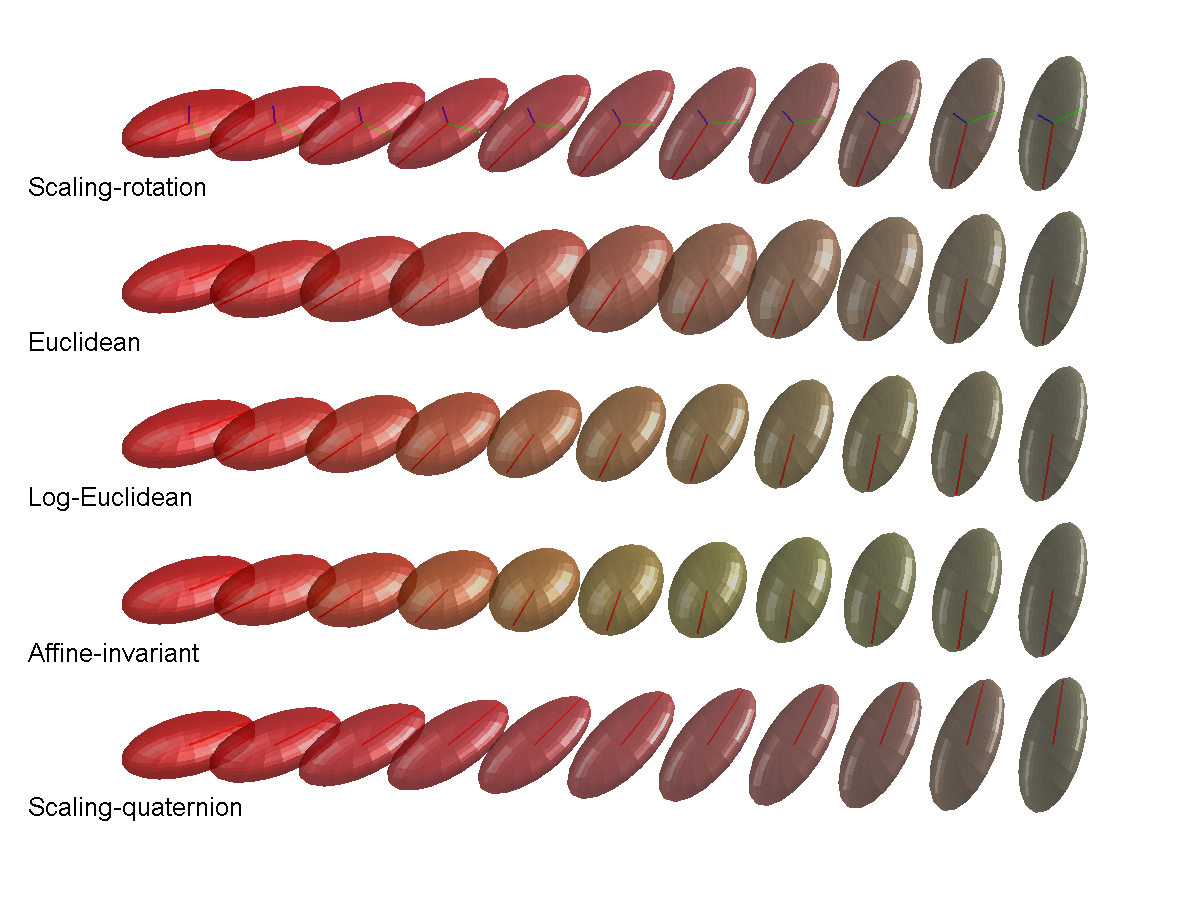}\\
  \vskip -0.2in
  \includegraphics[width=.9\textwidth]{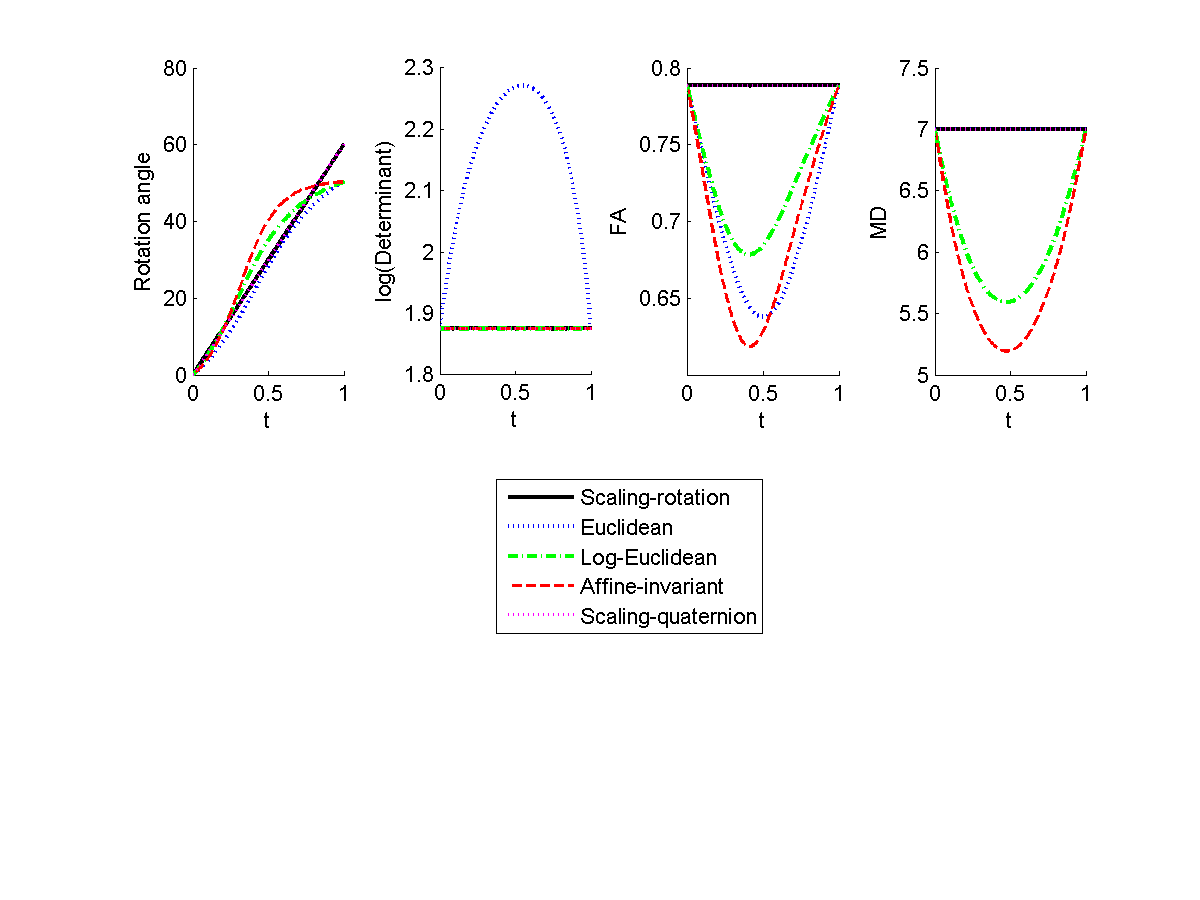}
  \vskip -1.0in
  \caption{Case 1. Pure rotation in $f_{SR}$.
    \label{1} }
\end{figure}

 The scaling--rotation interpolation in the top row of Fig.~\ref{1} provides a \emph{constant} pattern: constant rotational speed, determinant, fractional anisotropy, and mean diffusivity.

  \textbf{Case 2}: Pure scaling.
  \begin{align*}
  X &= \mbox{diag}(15,5,1),  \\
  Y &= \mbox{diag}(7,12,8).
  \end{align*}

\begin{figure}[tb]
 \centering
  \vskip -0.1in
  \includegraphics[width=.8\textwidth]{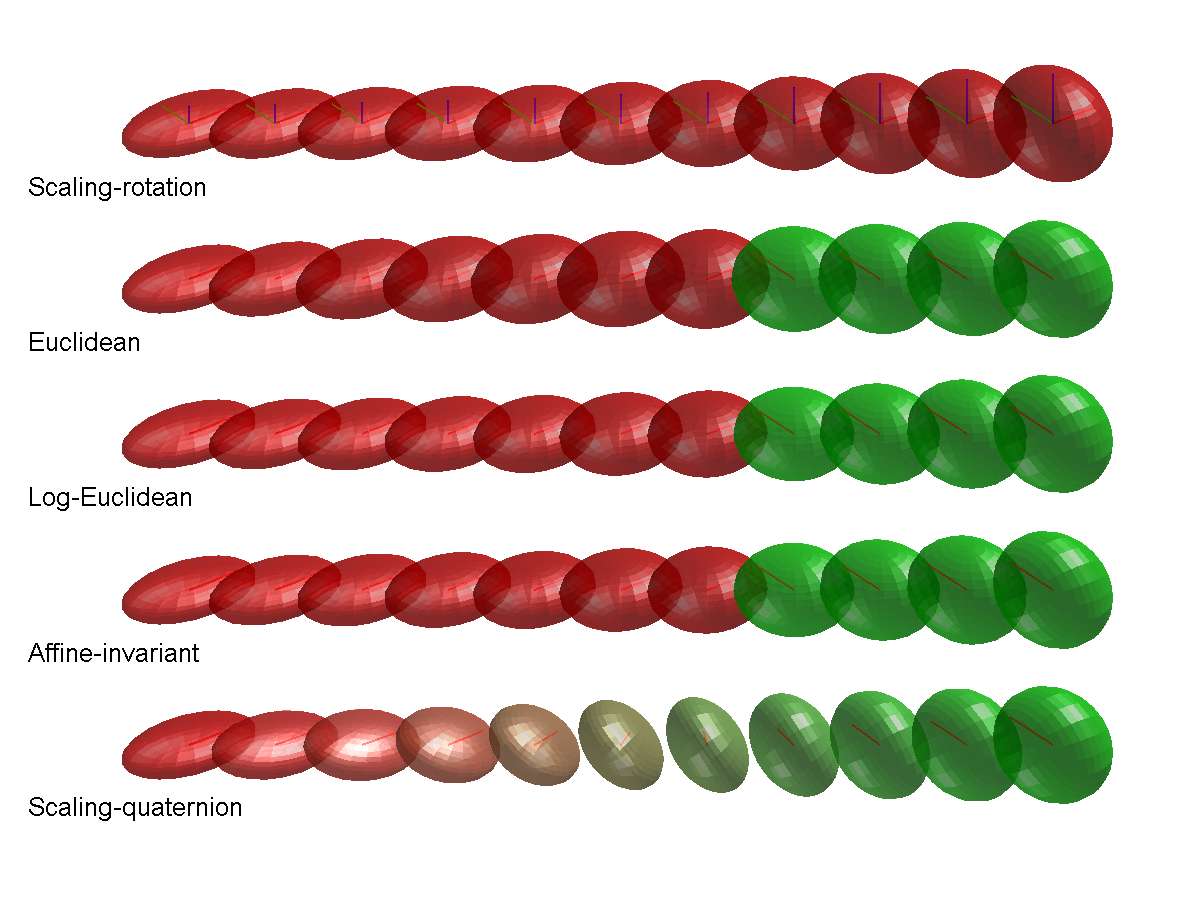}\\
  \vskip -0.2in
  \includegraphics[width=.9\textwidth]{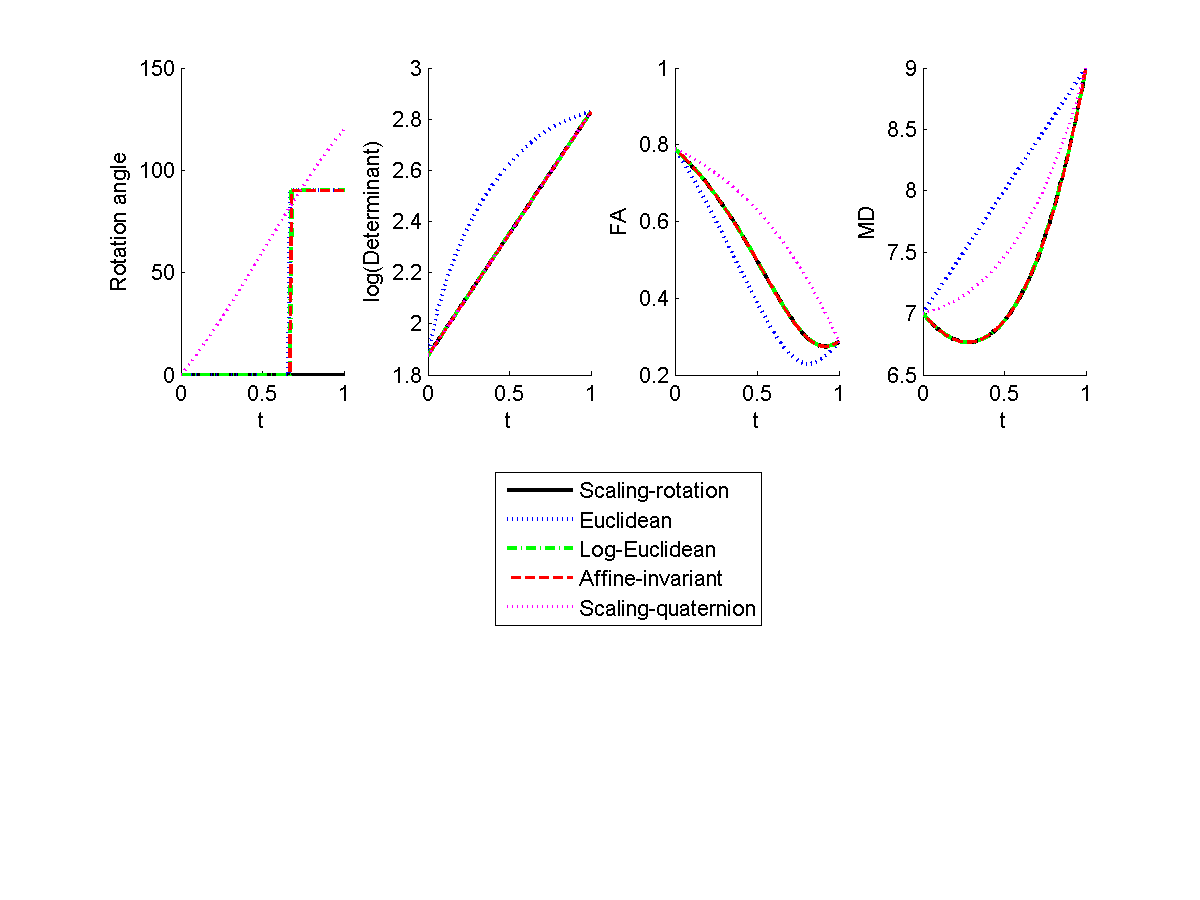}
  \vskip -1.1in
   \caption{Case 2. Pure scaling in $f_{SR}$.
  \label{2}}
\end{figure}
  It is evident in Fig. 2 that $f_{SR} \equiv f_{LE} \equiv f_{AI}$. The colors of $f_{LE}$ and $f_{AI}$ are sharply changed from red to green, because the change of the principal axis corresponding to the largest eigenvalue is not smooth. On the other hand, the principal axes of $f_{SR}$ do not change (since the interpolation is by pure scaling). Since  $f_{SR} \equiv f_{LE} \equiv f_{AI}$, the interpolation by $f_{SR}$ is only as good as the interpolations by $f_{LE}$. This is an example where these three interpolations all suffer from \emph{fattening and shrinking effects}.
  On the other hand, in this example it is \emph{not possible} for $f_{SQ}$ to interpolate only by scaling (because if so, the path would leave $\symp_*(3)$). The interpolation by $f_{SQ}$ thus involves the rotation of axes by $\pi/2$. While in this example $f_{SQ}$ is the only interpolation that does not suffer from the fattening and shrinking effects, for small enough scaling factor $k$ the scaling--rotation curve $f_{SR}$ for this pair $(X,Y)$ involves rotation, and does not exhibit fattening or shrinking.
  We elaborate on the choice of $k$ later in Section 2 of this online supplement.

  \textbf{Case 3}: A moderate mix of rotation and scaling.
  \begin{align*}
  X &= \mbox{diag}(15,5,1), \\
  Y &= R(\frac{\pi}{3}\av)  \mbox{diag}(9,12,8)    R(\frac{\pi}{3}\av)'.
  \end{align*}
\begin{figure}[tb]
 \centering
  \includegraphics[width=0.8\textwidth]{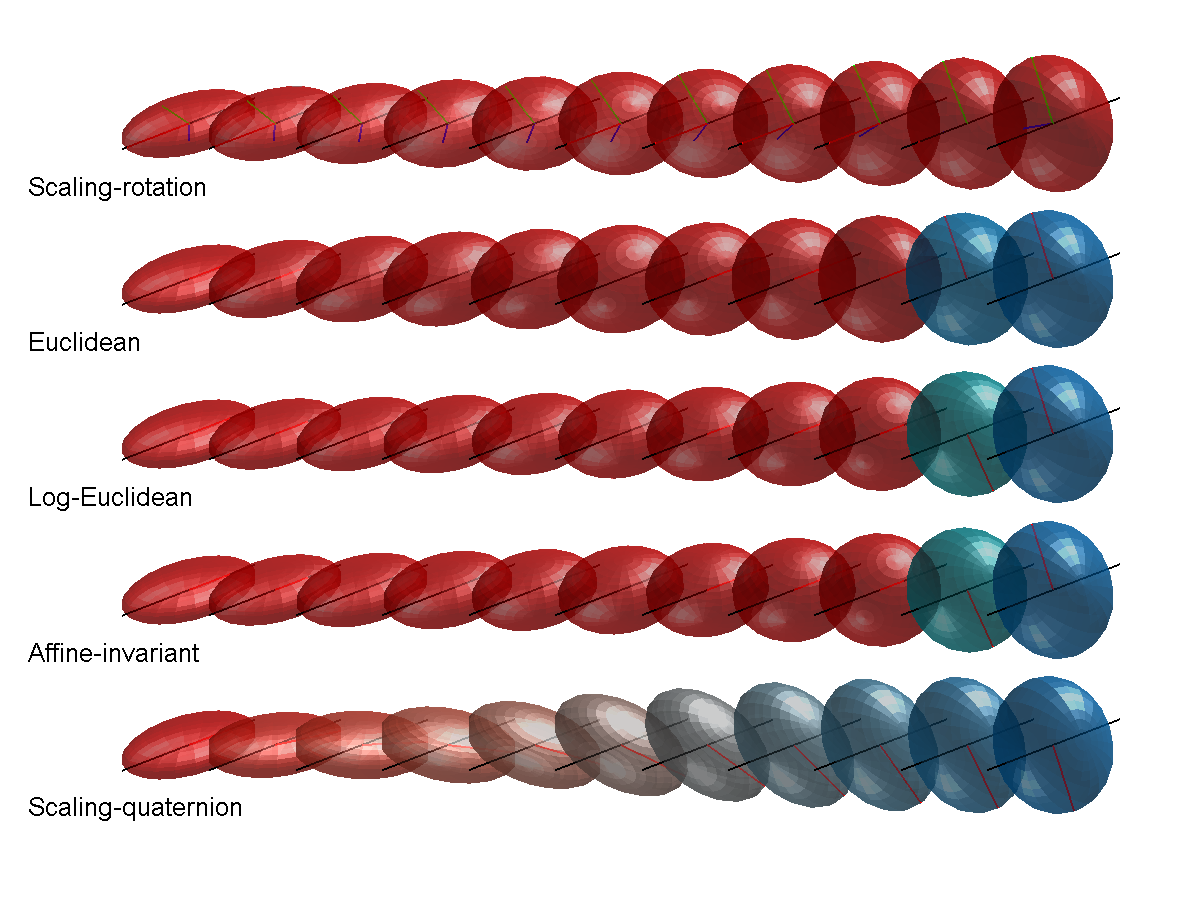}\\
  \includegraphics[width=0.9\textwidth]{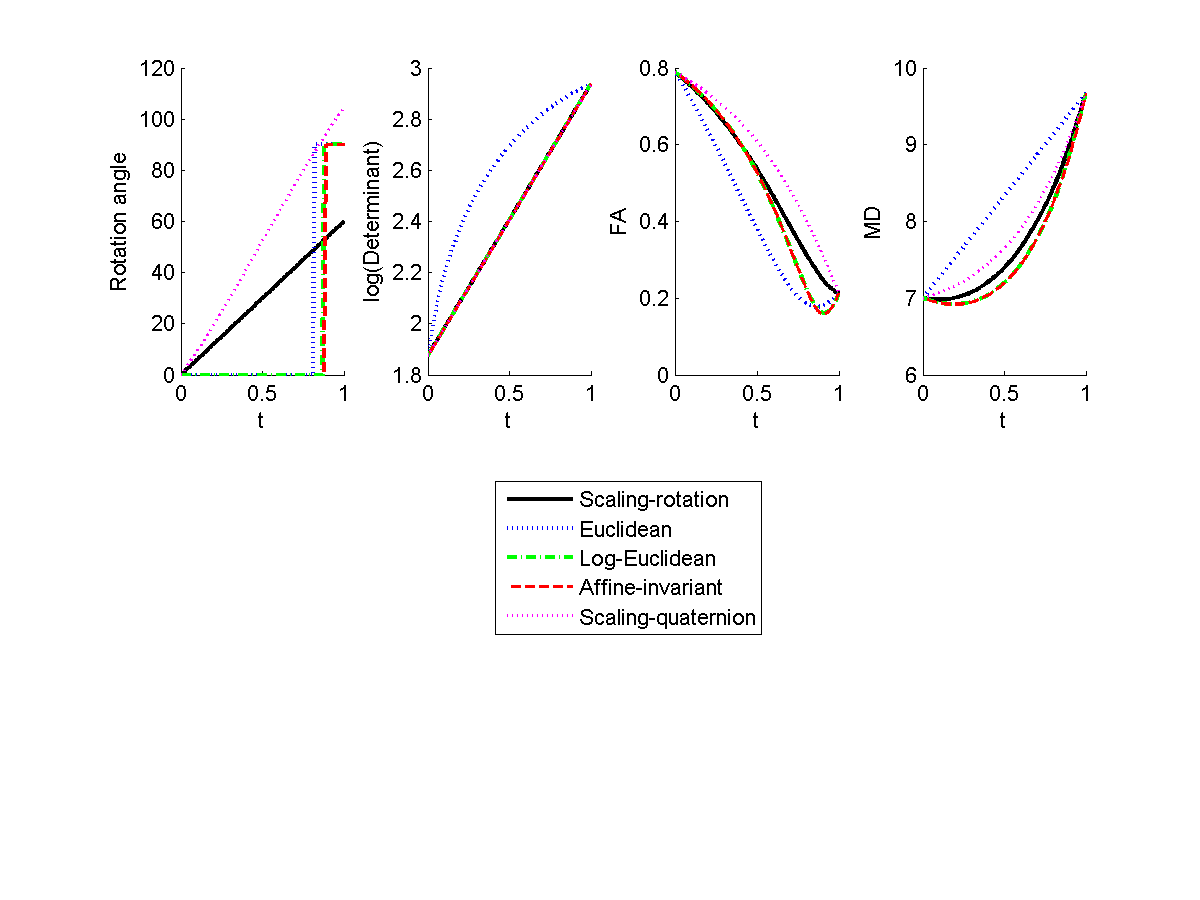}
  \vskip -1in
  \caption{Case 3. A moderate mix of rotation and scaling in $f_{SR}$.
  \label{3} }
\end{figure}

  The scaling--rotation interpolation in Fig.~\ref{3} provides a \emph{monotone} pattern for rotation angle, determinant, FA and MD.

\newpage
  \textbf{Case 4}: A moderate mix of rotation and scaling, but with $\mbox{MD}(X) = \mbox{MD}(Y)$.
  \begin{align*}
  X &= R(\frac{\pi}{6}\av)\mbox{diag}(1,15,4)R(\frac{\pi}{6}\av)',\\
  Y &= R(-\frac{\pi}{6}\av)  \mbox{diag}(2,10,8)    R(-\frac{\pi}{6}\av)'.
  \end{align*}
  \begin{figure}[tb]
   \centering
    \includegraphics[width=0.8\textwidth]{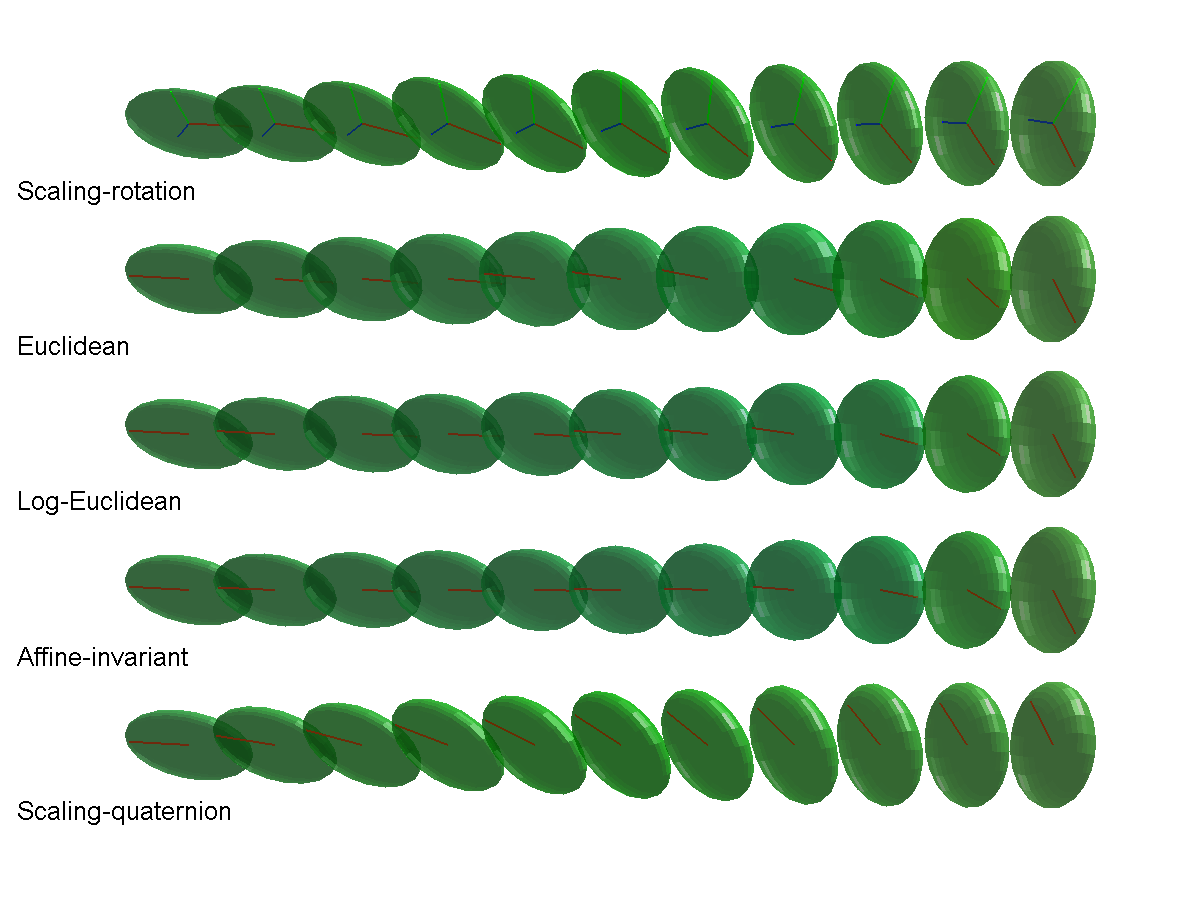}\\
    \includegraphics[width=0.9\textwidth]{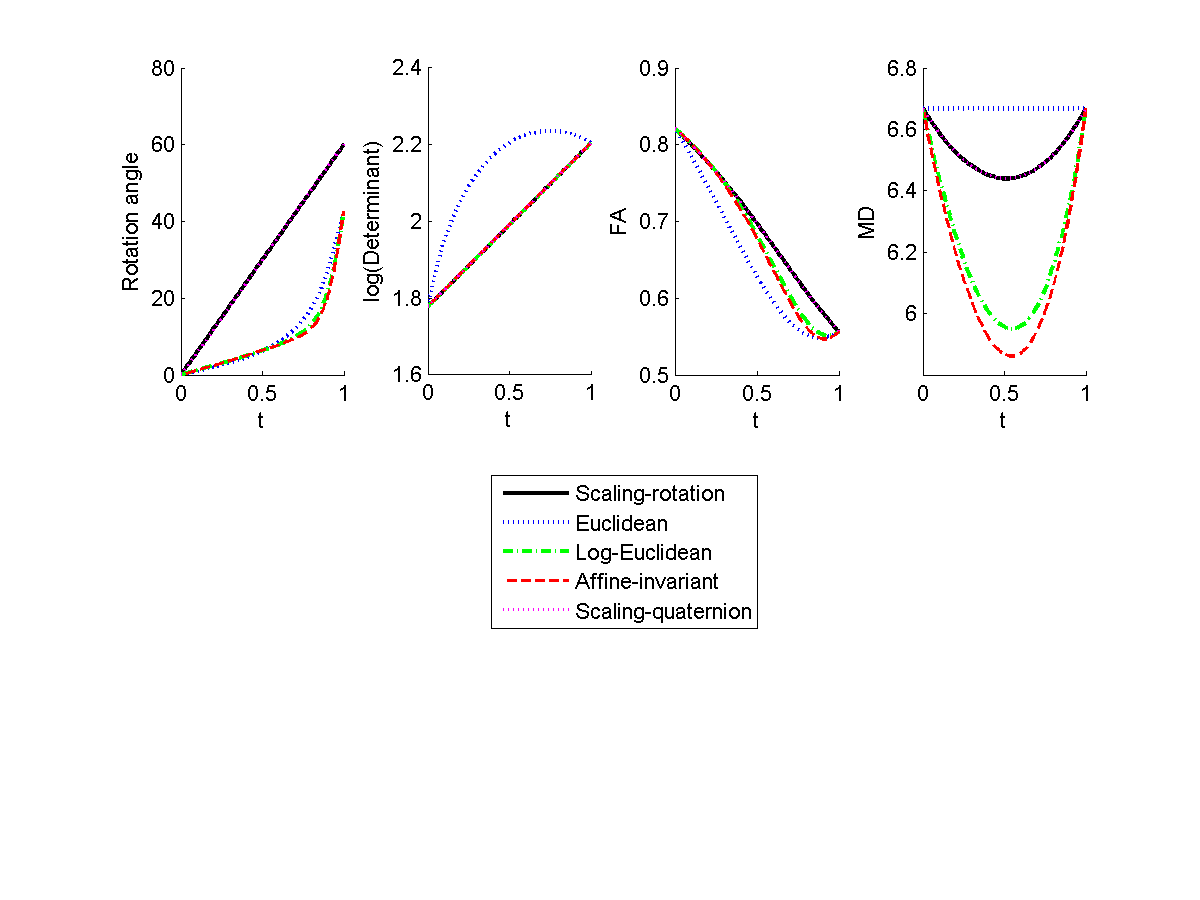}
    \vskip -1in
    \caption{Case 4. A moderate mix of rotation and scaling.
    \label{4}}
  \end{figure}
In Fig. \ref{4}, $f_{SR}$ does not exhibit monotone evolution of MD, but the \emph{shrinking effect} is  less  than in both $f_{LE}$ and $f_{AI}$, \textit{i.e.} MD($f_{SR}(t)$) is greater than both MD($f_{LE}(t)$) and MD($f_{AI}(t)$) for all $t \in (0,1)$.

\newpage
  \textbf{Case 5}: Departure from isotropy.
  \begin{align*}
  X &= \mbox{diag}(4,4,4) = 4 I_3,\\
  Y &= R(\frac{\pi}{3}\av)  \mbox{diag}(11,11,6)    R(\frac{\pi}{3}\av)'.
  \end{align*}

  \begin{figure}[tb]
   \centering
    \includegraphics[width=0.8\textwidth]{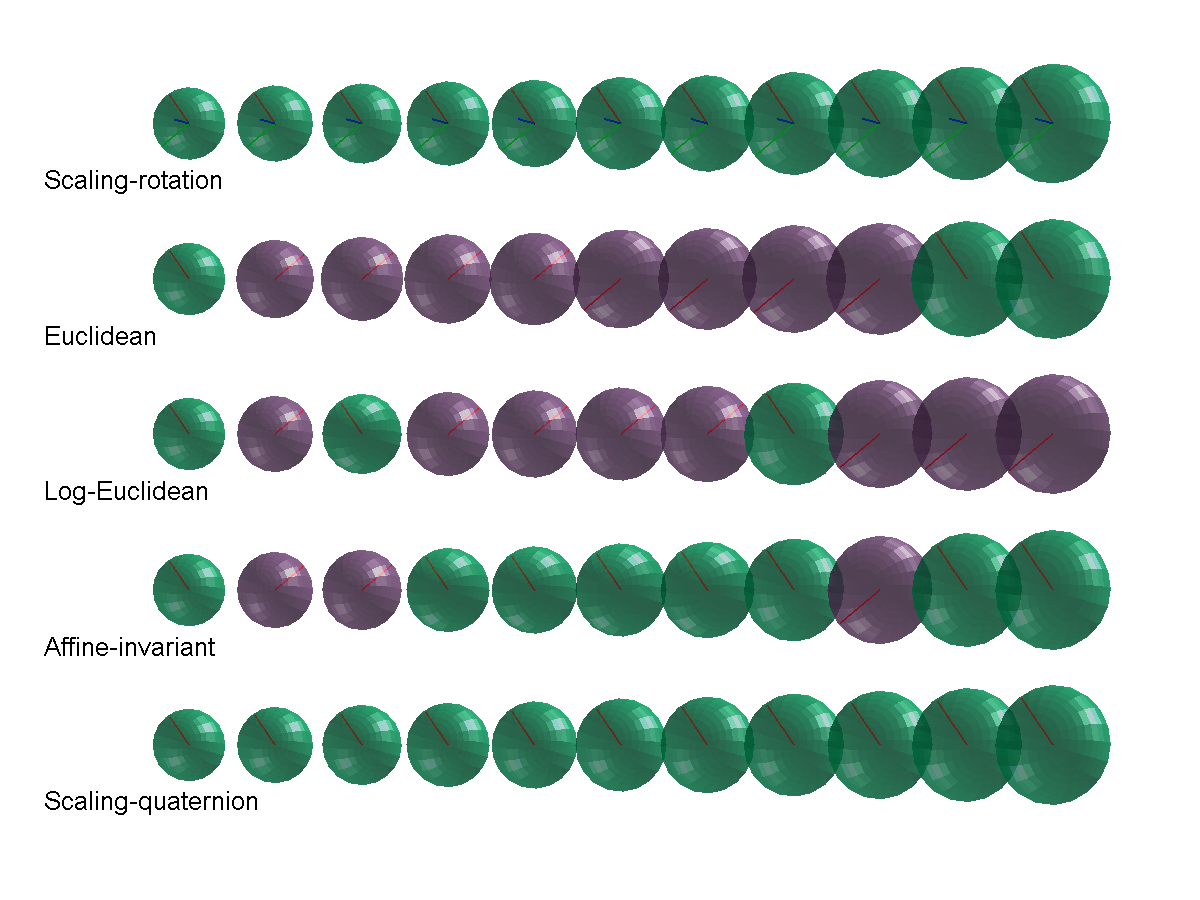}\\
    \includegraphics[width=0.9\textwidth]{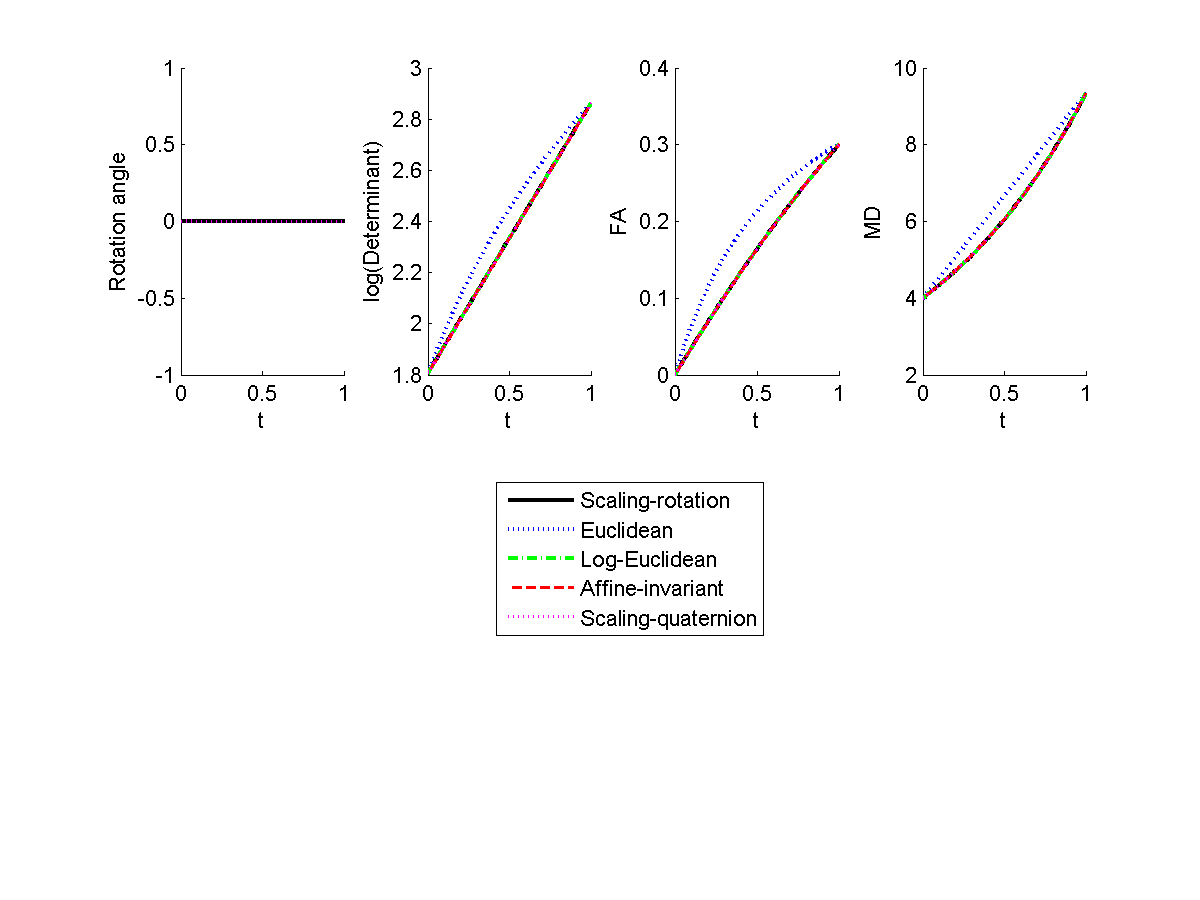}
    \vskip -1in
    \caption{Case 5:  Departure from  isotropy.
    \label{5}}
  \end{figure}

  The scaling--rotation interpolation $f_{SR}$ is of a pure scaling and thus $f_{SR} \equiv f_{LE} \equiv f_{AI}$ in this example. In the bottom left panel, we omit the rotation angles of $f_E,f_{LE}$ and $f_{AI}$, since the principal axes are not uniquely defined.
  The color-flips in $f_E,f_{LE}$ and $f_{AI}$ (shown in rows 2--4) are an artifact of our having made arbitrary choices of the principal axes.
The scaling--quaternion  interpolation ($f_{SQ}$) for equal-eigenvalue cases was not defined in \cite{collard2012anisotropy}. For this example, we chose the eigenvector matrix of $X$ to be the same as any given eigenvector matrix of $Y$, and defined $f_{SQ}(t) = (4 I_3)^{1-t} \Lambda^t$, where $\Lambda = \mbox{diag}(11,11,6)$.

\newpage
\section{Effect of the scaling factor $k$}
In connection with Section~5.2 of the main article, different choices of the scaling factor $k$ in the geodesic distance function

\renewcommand{\theequation}{\arabic{section}.\arabic{equation}}
\setcounter{section}{3}
\setcounter{equation}{3}
\setcounter{figure}{5}

\begin{align}
d^2\left((U,D),(V,\Lambda)  \right)
  =  k d_{{\rm SO}(p)}(U,V)^2 + d_{\Dc}(D_, \Lambda)^2, \quad k>0, \label{eq34}
\end{align}
 lead to different scaling--rotation distances $d_{\Sc\Rc} (X,Y)$ and different scaling--rotation interpolation between $X,Y \in \symp(p)$. Large values of $k$ tend to prevent the interpolation from including rotation. For smaller values of $k$, the interpolation tends to include rotation.

 We illustrate the effect of $k$ on the interpolations and their determinant, FA and MD. Consider interpolating between $X = \mbox{diag}(15,5,1)$ and $Y =  \mbox{diag}(7,12,8)$  (the same matrices used in Case 2 and Fig.~\ref{2} of Section 1 of this supplementary material).
 \begin{figure}[tb]
 \centering
  \includegraphics[width=0.8\textwidth]{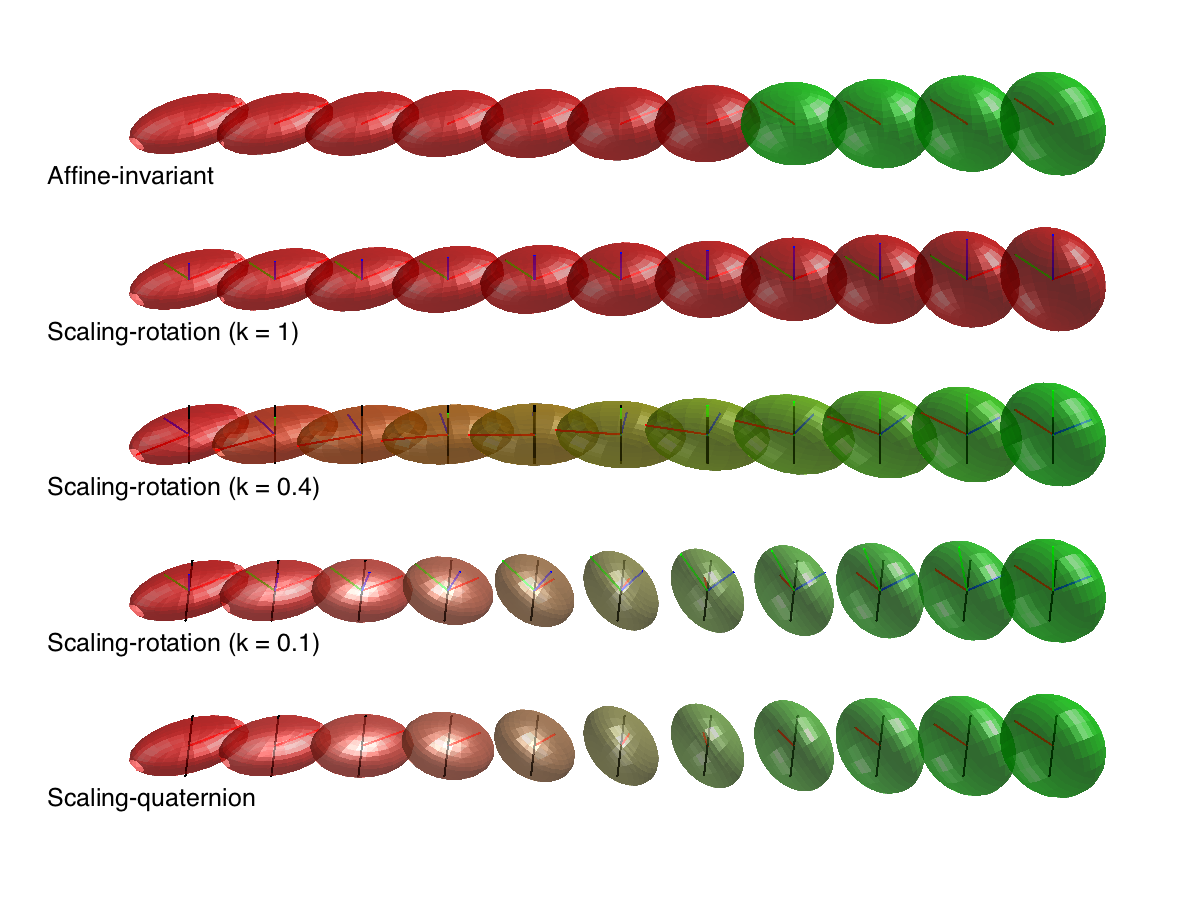}\\
  \vskip -0.2in
  \includegraphics[width=0.9\textwidth]{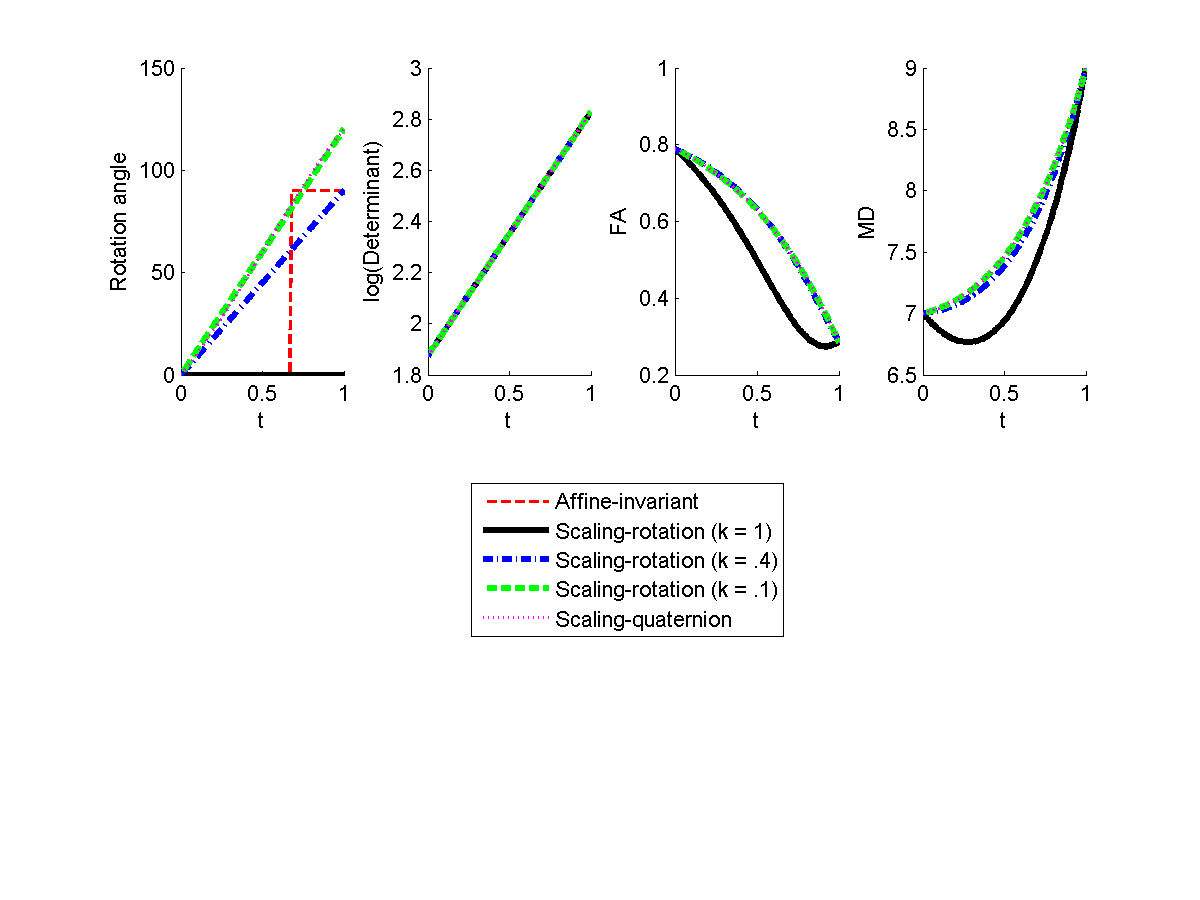}
  \vskip -1.0in
  \caption{Examples of  minimal scaling--rotation curves with the same endpoints but for different choices of $k$. For endpoints we use the same matrices used in Fig. 2 earlier in this supplement. Here, the log-Euclidean and affine-invariant interpolations are identical to each other.
  In this example they also coincide with the $k=1$ scaling-rotation curve, although this is not true in general.  The scaling-quaternion interpolation is by \cite{collard2012anisotropy}.
  \label{9}}
\end{figure}
 The scaling--rotation distance $d_{\Sc\Rc} (X,Y)$ is an increasing and continuous function of $k>0$, with discontinuities in the derivative appearing at approximately $k = 0.22$ and $k = 0.46$.
 The corresponding minimal scaling--rotation curves $\chi_{(k)}$ are different for different intervals to which the value of $k$ belongs. These are $(0,0.22)$, $(0.22,0.46)$, and $(0.46, \infty)$ approximately; in Fig.~\ref{9}, we chose $k = 1, 0.4, 0.1$ as representatives of these three intervals.
 For large $k$ ($k=1$ in Fig.~\ref{9}), the scaling--rotation curve $\chi_{(k)}$  involves only scaling and coincides with affine-invariant and log-Euclidean interpolations (see rows 1 and 2 of Fig.~\ref{9}). For smaller values of $k$ ($k = 0.4$ or $k = 0.1$ in Fig.~\ref{9}), the scaling--rotation curve $\chi_{(k)}$ is a mix of rotation and scaling. In this example, to the naked eye $\chi_{ (0.1)}$ appears virtually identical to  the scaling-quaternion interpolation, as shown in rows 4 and 5 of Fig.~\ref{9}. (However, the two interpolations are only approximately the same; they are not identical.)

\setcounter{section}{4}
\section{Advantageous use of {unordered} eigenvalues}
In connection with Section 5.2 of the main article, we provide an example where allowing unordered eigenvalues results in more desirable interpolations between SPD matrices. The method proposed by \cite{collard2012anisotropy} also measures  distance between SPD matrices by decomposing into rotations and scalings, but deals only with the case of distinct-and-ordered eigenvalues.

We compare distances and interpolations obtained by our scaling--rotation method and by the method of  \cite{collard2012anisotropy} between two SPD matrices whose corresponding ellipsoids are nearly oblate (\emph{i.e}. for each ellipsoid, the two longest principal axes are nearly equal).
The two SPD matrices we interpolate between are $(X_\epsilon,Y_\epsilon)$, where
  \begin{align}
  X_\epsilon &= \mbox{diag}(10+\epsilon,10-\epsilon,1), \nonumber \\
  Y_\epsilon &= R(\epsilon\frac{\pi}{4}  \ev_1)  \mbox{diag}(10-\epsilon,10+\epsilon,1) R(\epsilon\frac{\pi}{4}\ev_1)'  \label{eq1}
  \end{align}
for $\ev_1 =(1,0,0)'$ and $\epsilon \ge 0$.
For $\epsilon = 0$, $X_\epsilon = Y_\epsilon$. For small but non-zero $\epsilon$, the scaling--rotation distance and interpolation are computed from the unordered set of eigenvalues in (\ref{eq1}), with minimal amount of rotation ($\epsilon\frac{\pi}{4}$ in radians) and scalings $\log(\frac{ 10+\epsilon}{10-\epsilon})$. This leads to a  continuous transition to zero distance as $\epsilon\to 0$. See Fig.~\ref{fig:Scaling_Rotation_Paper_Fig2_supplmentR_Collard0} for a graph of $d_{\Sc\Rc}(X_\epsilon,Y_\epsilon)$.

In contrast,  assuming strictly \emph{ordered} eigenvalues results in more rotation than in the unordered case. In particular, by forcing the eigenvalues to be ordered, $X_\epsilon$ is compared with
$$
Y_\epsilon=R(\epsilon\frac{\pi}{4}  \ev_1) R(\frac{\pi}{2}  \ev_3)   X_\epsilon R(\frac{\pi}{2}  \ev_3)'  R(\epsilon\frac{\pi}{4}  \ev_1)' , \quad \ev_3 = (0,0,1)',
           $$
resulting in the excessive rotation by $\pi/2$ even when $X_\epsilon$ is arbitrarily close to $Y_\epsilon$. This phenomenon is illustrated  in Fig.~\ref{fig:Scaling_Rotation_Paper_Fig2_supplmentR_Collard0}.

\setcounter{figure}{6}
\begin{figure}[tb]
 \centering
  \includegraphics[width=0.7\textwidth]{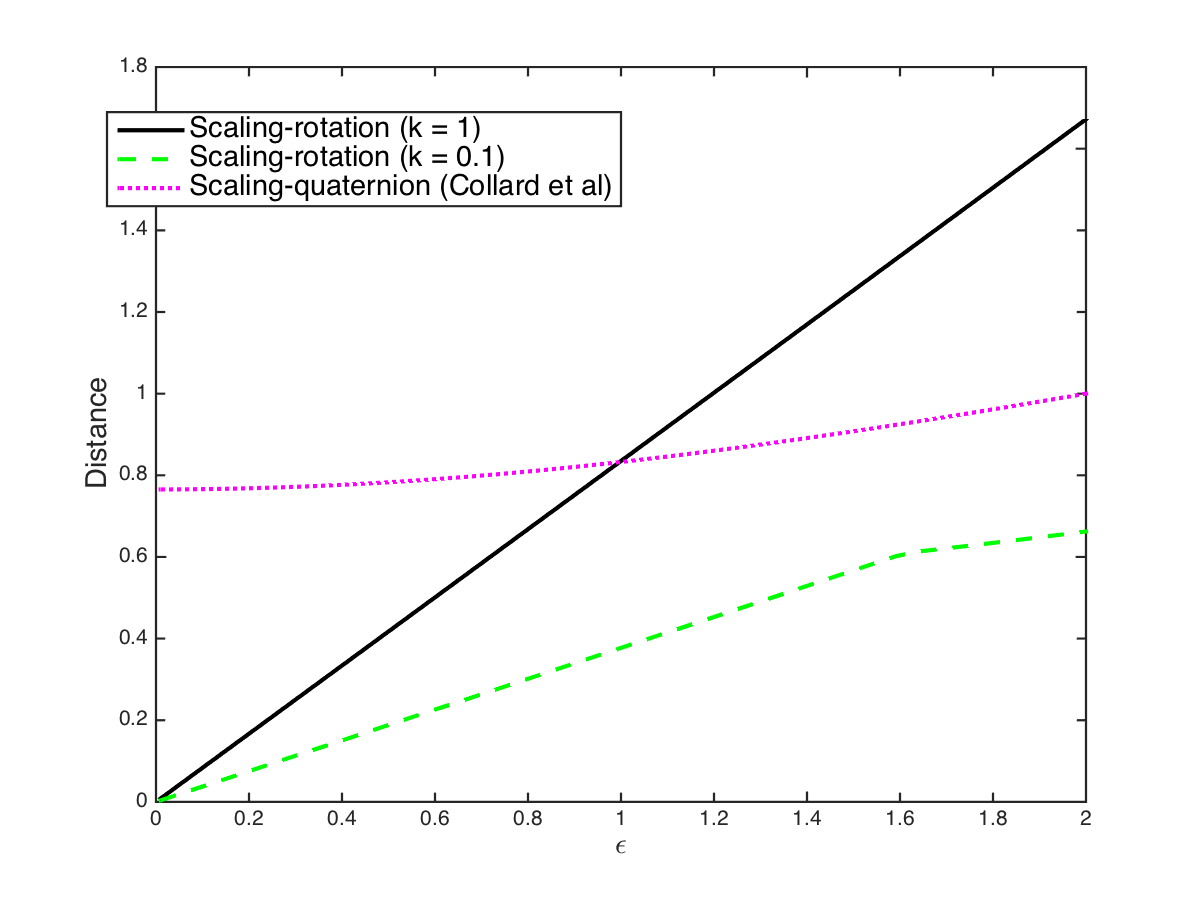}\\
   \caption{ Distances between $X_\epsilon$ and $Y_\epsilon$, for $0< \epsilon  < 2$; see (\ref{eq1}). The scaling--rotation distance (for both choices of $k = 1, 0.1$) is continuous on $\symp(3)$, while for the distance $d_C$ on $\symp_*(3)$ proposed in Collard et al., $d_C(X_\epsilon, Y_\epsilon)$ does not converge to zero as $\epsilon \to 0$. Note that $X_0 = Y_0 \notin \symp_*(3)$.
  \label{fig:Scaling_Rotation_Paper_Fig2_supplmentR_Collard0}}
\end{figure}

\setcounter{figure}{7}
\begin{figure}[htb]
 \centering
 \vskip -0.1in
  \includegraphics[width=0.8\textwidth]{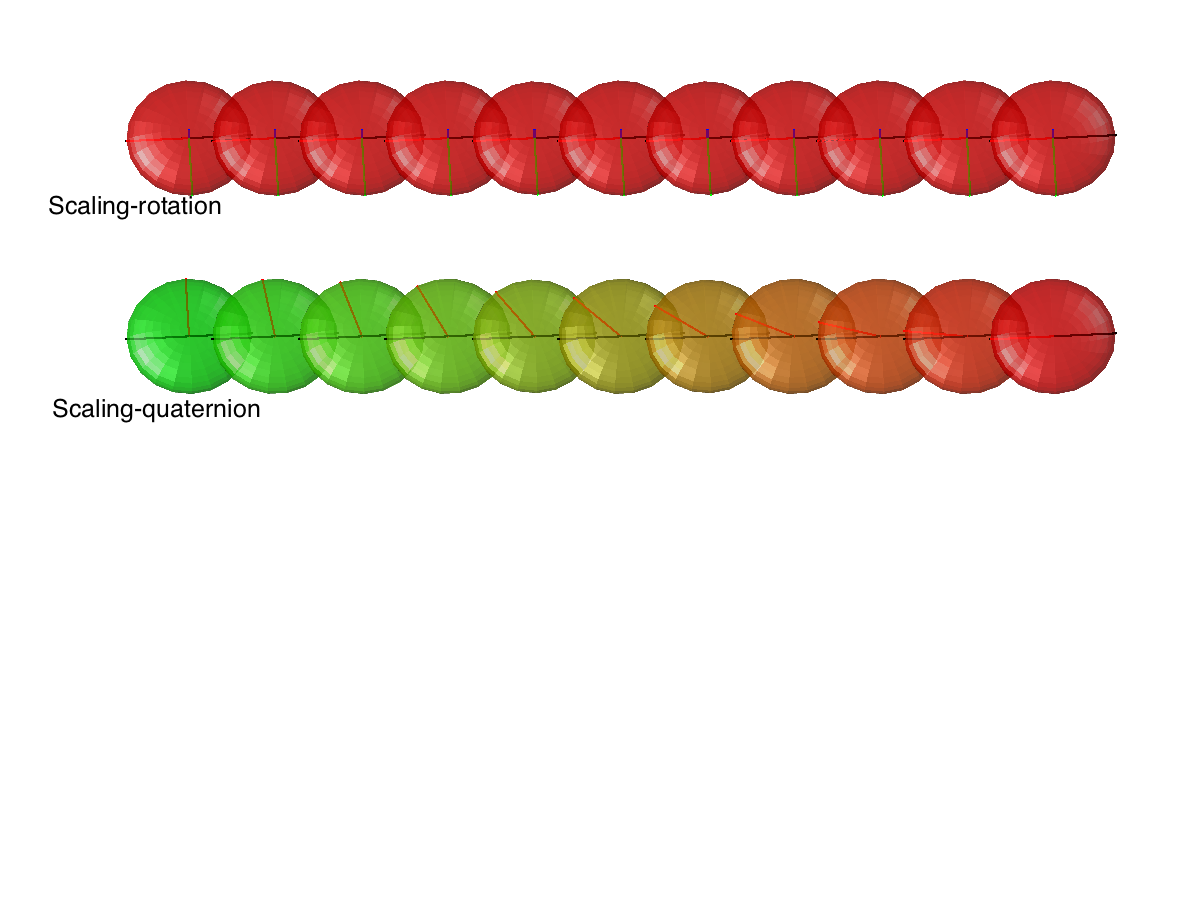}
   \vskip -1.4in
  \caption{Interpolations between $X_\epsilon$ and $Y_\epsilon$ for $\epsilon = 0.01$ (see (\ref{eq1})). Scaling--rotation interpolation in the first row shows almost negligible rotation, while scaling--quaternion interpolation by Collard et al. in the second row exhibits a large amount of rotation, illustrated by the changes in colors and by the swing of the red axis.
  \label{fig:Scaling_Rotation_Paper_Fig2_supplmentR_Collard}}
\end{figure}

One can devise a distance measure for $X_\epsilon = U_\epsilon D_\epsilon U_\epsilon'$, $Y_\epsilon = V_\epsilon\Lambda_\epsilon V_\epsilon' $ with a factor $k = k(D_\epsilon, \Lambda_\epsilon)$ for (\ref{eq34}) satisfying
\begin{align}\label{eq3}
\lim_{\epsilon \to 0 } k(D_\epsilon, \Lambda_\epsilon) = 0,
\end{align}
so that the distance between $X_\epsilon$ and $Y_\epsilon$ tends to zero  as  $\epsilon$ goes to zero.
 The factor function $k$ suggested by Collard et al. does not have this property in this example. Even if such a function $k$ satisfying (\ref{eq3}) is used, an unwanted rotation in the interpolation remains when eigenvalues are strictly ordered. In Fig.~\ref{fig:Scaling_Rotation_Paper_Fig2_supplmentR_Collard}, we see that although the initial and final ellipsoids are nearly identical to each other, the interpolation by Collard et al. exhibits a large amount of rotation, where our interpolation does not. Thus, in this example our interpolation may be viewed as more economical.


\end{document}